\documentclass{article}
\usepackage{graphicx}
\usepackage{float}
\usepackage{tikz} 
\usepackage{amsmath} 
\usepackage{bbm}
\usepackage{amsthm}
\usepackage{amssymb} 
\usepackage{subcaption}
\usepackage{fullpage}
\usepackage{accents}
\usepackage{orcidlink}

\definecolor{viola}{rgb}{0.3,0,0.7}
\definecolor{ciclamino}{rgb}{0.5,0,0.5}
\definecolor{blu}{rgb}{0,0,0.7}
\definecolor{rosso}{rgb}{0.9,0,0}

\def\pier #1{{\color{black}#1}}
\def\abramo #1{{\color{black}#1}}
\def\betti #1{{\color{black}#1}}

\def\mz #1{{\color{black}#1}}
\def\aa #1{{\color{black}#1}}
\def\pc #1{{\color{black}#1}}
\def\er #1{{\color{black}#1}}

\numberwithin{equation}{section}

\newtheorem{theorem}{Theorem}[section] 
\newtheorem{lemma}[theorem]{Lemma}

\newtheorem{remark}[theorem]{Remark}

\DeclareMathOperator*{\esssup}{ess\,sup}

\def\div {\nabla_x \cdot}
\renewcommand{\vec}{\boldsymbol}
\newcommand{\ubar}[1]{\underaccent{\bar}{#1}}

\begin{document}

\begin{center}
		
		{\LARGE \textbf{Global existence for  Consensus-Based Kinetic Models \\
		\vskip0.4cm
        for Image Segmentation}}

		\vskip0.5cm
		
		{\large\textsc{Abramo Agosti$^1$} \orcidlink{0000-0001-5706-3772}} \\
		{\normalsize e-mail: \texttt{abramo.agosti@unipv.it}} \\
		\vskip0.35cm
		
		{\large\textsc{Pierluigi Colli$^{1,2}$} \orcidlink{0000-0002-7921-5041}} \\
		{\normalsize e-mail: \texttt{pierluigi.colli@unipv.it}} \\
		\vskip0.35cm
		
		{\large\textsc{Elisabetta Rocca$^{1,2}$} \orcidlink{0000-0002-9930-907X}} \\
		{\normalsize e-mail: \texttt{elisabetta.rocca@unipv.it}} \\
		\vskip0.35cm
		
		{\large \textsc{Mattia Zanella$^1$} \orcidlink{0000-0001-8456-5866}} \\
		{\normalsize e-mail: \texttt{mattia.zanella@unipv.it}} \\
		\vskip0.35cm
		
		{\footnotesize $^1$Department of Mathematics ``F. Casorati'', University of Pavia, 27100 Pavia, Italy}
		\vskip0.1cm

        {\footnotesize$^2$Research Associate at the IMATI–C.N.R. Pavia, 27100 Pavia, Italy}
        \vskip0.1cm
		
	\end{center}



\begin{abstract}
In this article, we analyze a recently introduced kinetic model based on Hegselmann-Krause-type interaction dynamics  describing consensus-based image segmentation in the mean-field regime of a large number of particles, \aa{where the unknown particles density depends on both the spatial coordinates and the normalized gray level (feature) value}. The model is formulated as \aa{an evolutive} nonlocal partial differential equation featuring \aa{a non-local spatial drift term involving} a bounded confidence interaction kernel, depending on both spatial and feature proximity, and a \aa{spatial diffusion term degenerating in the feature variable} representing aleatoric uncertainties.

Establishing the well-posedness of this model on a bounded domain, which is essential for practical image processing, presents significant analytical challenges. These consist in the lack of uniform control of space derivatives of the solution due to interaction kernels \aa{with jump discontinuities} and degenerate \aa{spatial} diffusion, in the complexities of non-standard boundary terms under homogeneous Neumann conditions, and in the lack of compactness arising from the absence of derivatives with respect to the feature variable \aa{in the equation}.

To address these difficulties, we introduce a three-level regularization scheme. \pc{Specifically, we smooth the interaction kernel, apply a structural boundary regularization to the non-local interaction term in order to simplify the treatment of the boundary conditions, and incorporate} artificial diffusion both in the spatial and feature variables to resolve the degeneracy. We prove the well-posedness and the regularity of the solution of the regularized system via a finite difference scheme and study the asymptotic limit as the regularization parameters vanish, which leads to prove the existence of weak/distributional solutions of the original model. This approach significantly advances the mathematical theory of kinetic models for image segmentation by successfully extending the framework of distributional solutions to bounded domains and non-smooth interaction kernels \aa{with jump discontinuities}.
\vskip3mm
		
    
\vskip3mm
    
\end{abstract}

\section{Introduction}
Over the past few decades, a variety of computational strategies and mathematical approaches have been developed to address image segmentation problems. Among these, deep learning techniques based on neural networks have emerged as some of the most widely used methods in contemporary image segmentation tasks; see, e.g. \cite{Agosti_etal,NN,Sharma,Wang}. Although deep learning models excel in complex segmentation problems, their dependence on large annotated datasets remains a significant limitation, particularly in fields such as biomedical imaging, where data availability is scarce and manual labeling is both expensive and time-consuming.

A different approach is based on clustering methods, see e.g. \cite{CA,HS,JMF,Mittal}. These methods have been developed to group pixels with similar characteristics, effectively partitioning the image into distinct regions. Clustering-based methods offer an attractive alternative to deep learning techniques as they do not require supervised training and therefore can be used on small unlabeled datasets. In this direction, a kinetic approach for unsupervised clustering problems for image segmentation has been introduced in \cite{CPLFZ,HPV}. In these works, image segmentation has been connected to microscopic consensus-type models by considering the pixels of an image as composing an interacting system where each particle is characterized by its space position and a feature determining its gray level, and where each particle interacts with other particles that are sufficiently closed based on their distance in space and the proximity of their gray level, following an Hegselmann--Krause (HK) type nonlinear compromise interaction dynamics \cite{HK}\mz{, see also \cite{BL,Deffuant,D,MT,PTTZ,Toscani06} for extensions in consensus dynamics}. \mz{The HK interaction captures bounded-confidence dynamics in which agents adjust their state towards those of agents whose state lies within a prescribed confidence threshold, while more distant agents have no direct influence}. It results in the asymptotic formation of a finite number of clusters, which, \mz{in the context of image segmentation}, produce a segmentation mask by assigning to each cluster the mean of the gray levels of the particles that compose it and by applying a binary threshold \cite{CTZ,CPLFZ}. \mz{This asymptotic behavior has been extensively investigated and exploited for the segmentation of grayscale images  \cite{CTZ,CPLFZ,HPV,TGZ}. }

The mean-field limit of the aforementioned HK type microscopic dynamics, including random noise representing aleatoric-type uncertainties in the image acquisition process, is obtained in \cite{CPLFZ} and consists in the following nonlocal partial differential equation (PDE):
\begin{equation}
    \label{intro1}
    \partial_tf-\div \left(\Sigma(f)f\right)-\sigma D(c)\Delta_xf=0,
\end{equation}
valid in $\mathbb{R}^2\times [0,1]\times[0,T]$ and endowed with a proper initial condition, where $f:\mathbb{R}^2\times [0,1]\times [0,T]\to \mathbb{R}$ is the mean field limit of the empirical distribution of the HK model \er{and} $f=f(\mathbf{x},c,t)$ represents the density of particles around a point $\mathbf{x}\in \mathbb{R}^2$ with a normalized gray level value $c\in [0,1]$ at time $t\in [0,T]$, and the nonlocal interaction $\Sigma$ is defined as 
\begin{equation}
    \label{intro2}
\Sigma(f)\pier{(\vec{x},c,t)}=\int_0^1\int_{\mathbb{R}^2}P_{\Delta_1,\Delta_2}(\vec{x},\vec{x}^*,c,c^*)\left(\vec{x}-\vec{x}^*\right)f(\vec{x}^*,c^*,t)\,d\vec{x}^*dc^*,
\end{equation}
where the bounded confidence kernel has the following form:
\begin{align}
\label{intro3}
&P_{\Delta_{1},\Delta_{2}}(\vec{x},\vec{x}^*,c,c^*) =  \chi(|\mathbf{x} - \mathbf{x}^*| \leq \Delta_{1})\chi(|c - c^*| \leq \Delta_{2}).
\end{align}
\pc{Here, $\chi$ denotes the characteristic function} \mz{and $\Delta_1,\Delta_2>0$ are confidence thresholds}. \aa{Note that the drift and the diffusion terms in \eqref{intro1} involve only the spatial derivatives.} The parameter $\sigma>0$ in \eqref{intro1} is a measure for the amplitude of the random fluctuations in the system. Since aleatoric uncertainties are expected to appear only far away from the static feature’s
boundaries, a \er{non-constant} diffusion coefficient $D(c)=c(1-c)$, which is maximal at the center of the feature interval and \aa{degenerates at its endpoints}, \er{is} considered in \eqref{intro1}. The case without noise, i.e. $\sigma=0$, was firstly introduced and investigated in \cite{HPV}.

The aim of the present work is to study the well posedness and the regularity of solutions of \eqref{intro1}, defined over the bounded domain \er{$\Omega\times [0,1]\times [0,T]$, with $\Omega\subset \mathbb{R}^d$ and $d\leq 3$.} Even if typically in the kinetic modeling literature Fokker--Planck (FP) equations of type \eqref{intro1} are defined on the whole $\mathbb{R}^d$ or on the $d$-dimensional torus $\mathbb{T}^d$, considering a generic open bounded domain $\Omega\subset \mathbb{R}^d$ is more relevant in our context in view of the applications to image segmentation. This implies the need to define proper boundary conditions for \eqref{intro1} and to deal with boundary terms in the analysis of \eqref{intro1} originating from variational approaches. In the following we will consider homogeneous Neumann boundary conditions, and we will assume $\sigma=1$ for the ease of exposition.

We now review some analytical results available in the literature concerning models related to \eqref{intro1}. The authors of \cite{CHAZ} considered the FP equation associated to the noisy HK model defined on a one dimensional periodic spatial domain with no dependence on a feature variable and with constant diffusion\pc{.  They proved global well-posedness and derived regularity results for weak solutions by combining fixed-point arguments with parabolic regularity theory and classical energy and compactness methods.} A similar result was obtained in \cite{CHEN} for the same FP equation defined on the whole $\mathbb{R}$. We observe that the techniques employed in those works cannot be directly extended to \eqref{intro1}, which depends on the feature variable $c$ and is degenerate for $c=0,1$. We also observe that at the degeneracy points $c=0,1$ the equation \eqref{intro1} reduces to the transport equation associated to the HK model without noise. In \cite{CCR},
\pc{the well-posedness of a class of nonlocal transport equations associated with interacting particle systems, including the HK model as a particular case, was established in $\mathbb{R}^d$. The analysis has been carried out within a measure-valued framework, with solutions considered as probability measures.
The measure solution is characterized} in a distributional sense as the push-forward of the initial condition through the Lagrangian flow associated to the characteristics system, crucially assuming the local Lipschitz regularity of the interaction kernels. The existence of distributional solutions to general transport equations with rough \aa{(at least with Sobolev regularity $W^{1,1}$)} and divergence-free advection fields defined on $\mathbb{T}^d$ has been obtained in \cite{BCC} as the limit of vanishing viscosity schemes, proving the convergence of such schemes to the unique Lagrangian solution of the transport equation. In the latter case, the low integrability of the distributional solution does not guarantee uniqueness. 
We observe that the techniques employed in \cite{BCC,CCR} cannot be generalized to \eqref{intro1}, \aa{when specified for $c=0,1$}, since \aa{it is characterized by} interaction kernels \aa{with jump discontinuities which do not fit $W^{1,1}$ regularity}, \aa{since the advection field is} non solenoidal and moreover since the existence and regularity theory of characteristics employed therein is not directly applicable to the case with a bounded domain. 
\mz{A general theory for Fokker--Planck equations \aa{defined in $\mathbb{R}^d$} with irregular coefficients was developed by Le Bris and Lions \cite{LeBrisLions}, \aa{relying, as in \cite{BCC}, upon the DiPerna--Lions’ theory for linear transport equations with $W^{1,p}$ coefficients and bounded divergence for the advection field \cite{DPL}}. Their framework permits degenerate diffusion, but its results \aa{can not be extended to our case with bounded domain and with interaction kernels characterized by jump discontinuities. Moreover, differently from the standard FP equation, in equation \eqref{intro1}} diffusion acts only in the spatial variable vanishing at the endpoints of the feature interval, while the drift is a nonlocal functional of the evolving density. These features motivate the regularization and compactness approach developed below.}

We summarize here the main difficulties we have to tackle in the analysis of \eqref{intro1}:
\begin{itemize}
    \item we have to deal with non-smooth interaction kernels \aa{with jump discontinuities} and degenerate diffusion, which makes the derivation of a-priori estimates for \eqref{intro1} very challenging due to the lack of uniform control of $\nabla_x f$;
    \item we have to \er{handle} non-standard boundary conditions throughout the derivation of suitable a-priori estimates for \eqref{intro1};
    \item  \er{the drift and
the diffusion terms in \eqref{intro1} involve only the spatial derivatives}, which causes a lack of compactness for sequences of suitable approximations  employed to prove existence.
\end{itemize}
In our analysis \er{we overcome} these difficulties by properly regularizing equation \eqref{intro1} at three different levels: \er{we firstly} regularize the interaction kernel to be sufficiently smooth depending on a positive regularization parameter $\theta$; \pc{then,} we \er{introduce} a structural regularization of the non-local interaction term, depending on a positive parameter $\delta$, in such a way that it \er{vanishes on} $\partial \Omega$, which simplifies the treatment of the boundary conditions; \pc{finally,} we regularize the degeneracy of the diffusion term, adding artificial diffusion proportionally to a positive parameter $\sqrt{\epsilon}$ both in the $\mathbf{x}$ and in the $c$ variables. \er{We then prove} the well-posedness and the regularity of the solution of the regularized systems, employing a finite difference scheme and considering its convergence, and \er{we study} the limit as $\theta,\delta,\epsilon\to 0^+$. This convergence analysis \er{leads to the proof of} existence of distributional solutions of \eqref{intro1}, defined on $\Omega \times [0,1]\times [0,T]$. To the best of our knowledge, this is a novel result which goes beyond the current state of the art concerning the analysis of kinetic models for image segmentation. Letting $c=0,1$, these solutions generalize the notion of distributional solutions introduced in \cite{CCR} to the case of a bounded domain and a non-smooth interaction kernel \aa{with jump discontinuities}.

The paper is organized as follows. In Section $2$ we introduce the notation and the functional setting and we state the problem and its regularization. In Section $3$ we prove the well-posedness and the regularity of the solution of the regularized model. In Section $4$ we study the asymptotic limits as the regularization parameters go to zero.

\section{\pier{The problem and its regularization}}

\subsection{Notation and \pier{functional setting}}
Let $\Omega \subset \mathbb{R}^d$, $d\leq 3$, be an open bounded domain with smooth boundary. We generally indicate vectors with bold letters. 
Given a vector $\vec{v}\in \mathbb{R}^d$ and a tensor $\vec{T}\in \mathbb{R}^{d\times d}$, we indicate by $|\vec{v}|$ and $|\vec{T}|$ their Euclidean and Frobenius norms respectively. For a normed space $X$, the associated norm and seminorm\pier{, when applicable,}  are denoted by $\lVert \cdot \rVert_X$ and $\lvert \cdot \rvert_X$ respectively. The dual space of a Banach space $X$ is denoted by $X'$. Given an open bounded domain $\omega\subset \mathbb{R}^{d+1}$, which in the following will typically be $\Omega \times (0,1)$,
we denote by $L^p(\omega;K)$ and $W^{r,p}(\omega;K)$ the standard Lebesgue and Sobolev spaces of functions defined on $\omega$ with values in a set $K$, where $K$ may be $\mathbb{R}$ or a multiple power of $\mathbb{R}$. Given a function $f:\Omega\times (0,1)\to K$, $f=f(\vec{x},c)$ for $\vec{x}\in \Omega$, $c\in (0,1)$, we will indicate $\nabla f, \Delta f$ as differential operators acting on $f$ with respect to \pier{both} variables $\vec{x}$ and $c$, while $\nabla_x f,\div f,\Delta_xf$ will indicate differential operators in the variable $\vec{x}$ and $\partial_c f,\partial_c^2f$ differential operators in the variable $c$.
If $K\equiv \mathbb{R}$,
we simply write $L^p(\omega)$ and $W^{r,p}(\omega)$.  In the case $p = 2$, we use the notations $H^1(\omega) := W^{1,2}(\omega)$ and
$H^2(\omega) := W^{2,2}(\omega)$, and we denote by $(\cdot,\cdot)$ and $\lVert \cdot \rVert$ the $L^2(\omega)$ scalar product and induced norm between
functions with scalar or vectorial values. For $f\in L^1(\omega)$, we define $\bar{f}:=|\omega|^{-1}(f,1)$.
The duality pairing between $H^1(\omega; K)$ and $(H^1(\omega; K))'$ is denoted by \abramo{$< \cdot, \cdot>_{\omega}$. When $\omega \equiv \Omega \times (0,1)$, we will use the notation $< \cdot, \cdot>$.}

For any $1\leq p \leq \infty$, we denote by $L^p(0, T; V)$ the Bochner
space of functions from $(0,T)$ with values in a Banach space $V$. 
We will also use the notation $C^k(\overline{\Omega}; K)$ and $C^k ([0, T]; V)$, $k\geq 0$, for the spaces of continuously differentiable functions up to order $k$ defined on $\Omega$ with values in a set $K$ or
from $[0, T]$ to the space $V$, respectively.

In the following, $C$ denotes a generic positive constant independent of the unknown variables, the discretization and regularization parameters\pier{. Its value may} change from line to line. \pier{On the other hand,} $C_1, C _2 , \dots$ indicate generic positive constants whose particular value must be tracked through the calculations. \pier{In addition,}
$C(a, b, . . . )$ denotes a constant depending on the nonnegative parameters $a, b, \dots$.

\subsection{\pier{Statement of the problem and regularization}}
Given $\Omega\subset \mathbb{R}^{d}$, $d\leq 3$, and $T>0$, we consider the problem of finding a solution $f:\Omega\times [0,1]\times[0,T]\to \mathbb{R}$ to the equation
\begin{equation}
    \label{wp1}
    \partial_tf-\div \left(\Sigma(f)f\right)-D(c)\Delta_xf=0,
\end{equation}
valid in $\Omega\times [0,1]\times[0,T]$, endowed with homogeneous Neumann boundary conditions
\begin{equation}
    \label{wp2}
    \Sigma(f)f\cdot \vec{n}+D(c)\nabla_xf\cdot \vec{n}=0\quad \text{on} \; \partial \Omega\times [0,1]\times[0,T],
\end{equation}
\betti{where $\vec{n}$ denotes the outer normal vector to $\partial\Omega$}, and with the initial condition
\begin{equation}
    \label{wp3}
    f(\vec{x},c,0)=f_0(\vec{x},c)\quad \forall \vec{x} \in \Omega, c\in [0,1],
\end{equation}
with $f_0\geq 0$, $\int_0^1\int_{\Omega}f_0d\vec{x}dc=1$.
\begin{remark}
\label{rem1}
Due to the fact that $f_0(\vec{x},c)$ represents the number of agents/particles per unit volume centered in $\vec{x}$ with feature $c$, \pier{it is} natural to assume that $f_0\in L^{\infty}(\Omega\times (0,1))$. In the following we will also consider other regularity classes for the initial condition.
\end{remark}
\noindent
We consider the following form for the mobility:
\begin{equation}
    \label{wp4}
    D(c)=c^a(1-c)^b, 
\end{equation}
where $a,b\in \mathbb{R}^+$. Also, we consider here the following form of nonlocality $\Sigma(f):\Omega \times [0,1]\times [0,T]\to \mathbb{R}^d$
\begin{equation}
    \label{wp5}
\Sigma(f)\pier{(\vec{x},c,t)}=\int_0^1\int_{\Omega}P_{\Delta_1,\Delta_2}(\vec{x},\vec{x}^*,c,c^*)\left(\vec{x}-\vec{x}^*\right)f(\vec{x}^*,c^*,t)\,d\vec{x}^*dc^*,
\end{equation}
where the bounded confidence kernel is given in \eqref{intro3}.
 We observe that, since the integration in \eqref{wp5} is over the bounded domain $\Omega \times [0,1]$, the integration variables $\mathbf{x}^*$ and $c^*$ are confined to vary in the sets $\mathbf{x}^*\in \mathcal{B}_{\Delta_1}(\mathbf{x})\cap \Omega$, where $\mathcal{B}_{\Delta_1}(\mathbf{x})$ is the ball centered in $\mathbf{x}$ with radius $\Delta_1$, and $c^*\in (\max(0,c-\Delta_2),\min(1,c+\Delta_2))$.
This represents the fact that a particle with position $\vec{x}^*$ and feature $c^*$ interacts with the particles in its spherical neighborhood of radius $\Delta_1$ in space and $\Delta_2$ in the feature interval up to the boundary of the domain, and cannot interact with particles outside of the domain $\Omega \times [0,1]$. Hence, \eqref{wp5} can be rewritten as
\begin{equation}
    \label{wp5b}
\Sigma(f)=\int_{\max(0,c-\Delta_2)}^{\min(1,c+\Delta_2)}\int_{\mathcal{B}_{\Delta_1}(\vec{x})\cap\Omega}\left(\vec{x}-\vec{x}^*\right)f(\vec{x}^*,c^*,t)\,d\vec{x}^*dc^*,
\end{equation}
which has the advantage that the integrand function is not explicitly written in terms of the non-smooth characteristic functions.  

Given a function $f:\Omega\times [0,1]\times [0,T]\to \mathbb{R}$, we indicate as $f(t):\Omega\times [0,1]\to \mathbb{R}$ the function $f(t)(\vec{x},c):=f(\vec{x},c,t)$ for a given $t\in [0,T]$.
The operator $\Sigma$ defined in \eqref{wp5} has the following properties.

\begin{lemma}
\label{lem1}
For fixed $t\in [0,T]$, the operator $\Sigma$ is a bounded linear operator which maps $L^1(\Omega \times [0,1])$ in $L^{\infty}\left(\Omega \times [0,1];\mathbb{R}^d\right)$, and, for any $f(t)\in L^1(\Omega \times [0,1])$, it holds that
\begin{itemize}
    \item[(i)] $|\Sigma(f(t))|\leq \Delta_1\lVert f(t)\rVert_{L^1(\Omega \times [0,1])}.$ 
\end{itemize}
Moreover, \abramo{for any $f(t)$ such that $f(\cdot,c,t)\in L^1(\Omega)$ for a.e. $c\in (0,1)$ and $\int_{\Omega}|f(x^*,\cdot,t)|\,dx^*\in L^{\infty}(0,1)$, it holds that
\begin{itemize}
    \item[(ii)] $\lVert \partial_c \Sigma(f(t))\rVert_{L^{\infty}(\Omega\times (0,1);\mathbb{R}^d)}\leq C\left\lVert\int_{\Omega}|f(x^*,\cdot,t)|\,dx^*\right\rVert_{L^{\infty}(0,1)}.$
\end{itemize}}
\end{lemma}
\begin{proof}
From the definitions \eqref{wp5}-\eqref{intro3} \pier{it is} easy to deduce that $\Sigma(f)$ satisfies
\[
|\Sigma(f(t))|\leq \Delta_1\int_0^1\int_{\Omega}|f(\vec{x}^*,c^*,t)|\,d\vec{x}^*dc^*,
\]
hence $(i)$ holds and in particular
\[
\lVert \Sigma(f(t))\rVert_{L^{\infty}(\Omega\times (0,1);\mathbb{R}^d)}
\leq  \Delta_1\lVert f(t)\rVert_{L^1(\Omega\times[0,1])},
\]
which implies that $\Sigma$ maps $L^1(\Omega \times [0,1])$ in $L^{\infty}\left(\Omega \times [0,1];\mathbb{R}^3\right)$. We now prove $(ii)$, which will be useful in the forthcoming sections.
Taking the derivative with respect to $c$ of \eqref{wp5b}, we obtain that
\begin{align*}
&\partial_c\Sigma(f(t))=H(1-c-\Delta_2)\int_{\mathcal{B}_{\Delta_1}(\vec{x})\cap\Omega}\left(\vec{x}-\vec{x}^*\right)f(\vec{x}^*,\min(c+\Delta_2,1),t)\,d\vec{x}^*\\
& \qquad -H(c-\Delta_2)\int_{\mathcal{B}_{\Delta_1}(\vec{x})\cap\Omega}\left(\vec{x}-\vec{x}^*\right)f(\vec{x}^*,\max(c-\Delta_2,0),t)\,d\vec{x}^*
\end{align*}
where $H(\cdot)$ is the Heaviside function.
\abramo{Taking the Euclidean norm on both sides, we easily \pier{infer} that
\begin{equation*}
\,|\partial_c\Sigma(f(t))|\leq \Delta_1\,\int_{\mathcal{B}_{\Delta_1}(\vec{x})\cap\Omega}|f(\vec{x}^*,c+\Delta_2,t)|\,d\vec{x}^*+\Delta_1\,\int_{\mathcal{B}_{\Delta_1}(\vec{x})\cap\Omega}|f(\vec{x}^*,c-\Delta_2,t)|\,d\vec{x}^*,
\end{equation*}
from which we \pier{obtain} $(ii)$.}
\end{proof}

\subsection{Regularized model}
We now introduce \pier{suitable} regularizations of the nonlocal \pier{and} diffusive terms in \eqref{wp1}. \pier{This allows us to establish existence and regularity results for solutions to the regularized problems. Subsequently, by analyzing the behavior of these solutions as the regularization parameters tend to zero, we prove the existence of solutions to the original model through a limiting argument.}

We first introduce a structural regularization of the operator $\Sigma$, depending on a positive parameter $\delta>0$, in such a way that its regularized version $\Sigma^{\delta}$ satisfies the property that $\Sigma^{\delta}(f)|_{\partial{\Omega}}=\mathbf{0}$ for any $\delta>0$. As will become clear later, this will be used to simplify the boundary conditions \eqref{wp2}. 
In order to proceed, let us introduce a function $G^{\delta}\in C(\bar{\Omega})\cap W^{2,\infty}(\Omega)$ such that:
\begin{itemize}
\item[(j)] $G^{\delta}(\vec{x})\equiv 1$ on $\Omega\setminus \Omega_{\delta}$, where $\Omega_{\delta}:=\{\vec{x}\in \Omega: \text{dist}(\vec{x},\partial \Omega)\leq \delta\}$;
\item[(jj)] $|G^{\delta}(\vec{x})|\leq 1$ on $\Omega$ and $G^{\delta}|_{\partial \Omega}=0$ for any $\delta>0$ ;
\item[(jjj)]$\lVert \nabla G^{\delta} \rVert_{L^{\infty}(\Omega_{\delta};\mathbb{R}^{d})}\leq \frac{C}{\delta}$ and $\lVert \nabla \nabla G^{\delta} \rVert_{L^{\infty}(\Omega_{\delta};\mathbb{R}^{d\times d})}\leq \frac{C}{\delta^2}$;
\item[(jv)] $G^{\delta}\to 1$ a.e. in $\Omega$ as $\delta \to 0$.
\end{itemize}
We observe that \betti{(jj) and (jv)} imply that
\begin{equation}
    \label{gdws}
    G^{\delta}\overset{\ast}\rightharpoonup 1\;\; \text{in}\;\; L^{\infty}(\Omega)\;\; \text{as}\;\;\delta \to 0.
\end{equation}
Indeed, for any $g\in L^1(\Omega)$, thanks to $(jj)$, $(jv)$ and the Lebesgue dominated convergence theorem \pier{it is} easy to \pier{see} that
\[
\lim_{\delta \to o}\int_{\Omega}(G^{\delta}(\vec{x})-1)g(\vec{x})\,d\vec{x}=0.
\]
Property \eqref{gdws} will be fundamental in the process of passing to the limit as $\delta \to 0$.
\newline
For instance, in the case $\Omega=[-L,L]^3$, for a given $L>0$, such a family of $G^{\delta}$ can be explicitly constructed e.g. by introducing the function
\[
h^{\delta}(s):=
\begin{cases}
-\frac{1}{\delta^2}(s+L)^2+\frac{2}{\delta}(s+L) \quad \text{for}\;\; -L\leq s \leq -L+\delta;\\
1 \quad \quad \quad \quad \quad \quad \quad \quad \quad \quad \;\;\, \, \text{for}\;\; -L+\delta\leq s \leq L- \delta;\\
-\frac{1}{\delta^2}(L-s)^2+\frac{2}{\delta}(L-s) \quad \text{for}\;\; L-\delta\leq s \leq L;
\end{cases}
\]
and setting $G^{\delta}(\vec{x})=h^{\delta}(\vec{x})h^{\delta}(y)h^{\delta}(z)$. The properties (j)-(jv) of $G^{\delta}$ are straightforwardly satisfied by the latter construction.
We finally define 
\begin{equation}
    \label{sigmadelta}
    \Sigma^{\delta}(f):=G^{\delta}\Sigma(f).
\end{equation}
Thanks to the properties of $G^{\delta}(\cdot)$ and repeating the arguments which lead to Lemma~\ref{lem1}, it is easy to deduce the following properties for $\Sigma^{\delta}$\pier{. Since their proofs are straightforward, we omit them.}
\begin{lemma}
\label{lem1d}
For fixed $t\in [0,T]$, the operator $\Sigma^{\delta}$ is a bounded linear operator which maps $L^1(\Omega \times [0,1])$ in $L^{\infty}\left(\Omega \times [0,1];\mathbb{R}^d\right)$, and
\begin{itemize}
    \item[(i)] $\Sigma^{\delta}(f)|_{\partial \Omega}=\mathbf{0}$;
    \item[(ii)] $|\Sigma^{\delta}(f(t))|\leq \Delta_1\lVert f(t)\rVert_{L^1(\Omega \times [0,1])}\pier{\quad\hbox{for any}\quad} f(t)\in L^1(\Omega \times [0,1]).$ 
\end{itemize}
Moreover, \abramo{for any $f(t)$ such that $f(\cdot,c,t)\in L^1(\Omega)$ for a.e. $c\in (0,1)$ and $\int_{\Omega}|f(x^*,\cdot,t)|\,dx^*\in L^{\infty}(0,1)$, it holds that
\begin{itemize}
    \item[(iii)] $\lVert \partial_c \Sigma^{\delta}(f(t))\rVert_{L^{\infty}(\Omega\times (0,1);\mathbb{R}^d)}\leq C\left\lVert\int_{\Omega}|f(x^*,\cdot,t)|\,dx^*\right\rVert_{L^{\infty}(0,1)}.$
\end{itemize}}
\end{lemma}

We then regularize the bounded confidence kernel ${\chi}_{\Delta_1}(\cdot):=\chi(|\cdot|\leq \Delta_1)$ by introducing a \pc{regularization} $\chi_{\theta,\Delta_1}$, depending on a parameter $\theta>0$, such that:
\begin{itemize}
\item[--] \pc{$\chi_{\theta,\Delta_1}$ is smooth:} 
\begin{equation}
\label{wp12}
\chi_{\theta,\Delta_1}\in C_c^2(K);
\end{equation}
\item[--] \pc{The convergenge property holds true:}
\begin{equation}
\label{wp13}
\chi_{\theta,\Delta_1}(\cdot)\to {\chi}_{\Delta_1}(\cdot) \quad \text{\abramo{a.e. in $K$ as}}\;\; \theta\to 0;
\end{equation}
\item[--] There exists a constant $C>0$ independent \pier{of} $\theta$ such that, for all $s\in K$, we have that
\begin{equation}
\label{wp14}
0\leq \chi_{\theta,\Delta_1}(s)\leq 1, \quad \left|\nabla \chi_{\theta,\Delta_1}(s)\right|\leq \frac{C}{\theta}, \quad \left|\nabla \nabla \chi_{\theta,\Delta_1}(s)\right|\leq \frac{C}{\theta^2}.
\end{equation}
\end{itemize}
Here\pc{,} $K$ may be $\mathbb{R}$ or $\mathbb{R}^d$.
In order to do this we introduce the following smooth regularization of the Heaviside function:
\[
H_{\theta}(s):=
\begin{cases}
\begin{aligned}
    & 0, && s< 0,\\
    & 10\frac{s^3}{\theta^3}-15\frac{s^4}{\theta^4}+6\frac{s^5}{\theta^5}, && 0\leq s\leq \theta,\\
    & 1, && s> \theta,
\end{aligned}
\end{cases}
\]
which satisfies $H_{\theta}\in C^2(\mathbb{R})$, \abramo{$H_{\theta}(\cdot)\to H(\cdot)$ a.e. in $\mathbb{R}$ as $\theta\to 0$}, and
\[\left|\frac{d}{ds}H_{\theta}(s)\right|\leq \frac{C}{\theta}, \quad \left|\frac{d^2}{ds^2}H_{\theta}(s)\right|\leq \frac{C}{\theta^2}.\]
We also need to introduce for later purposes the following convex smooth regularization of the absolute-value function:
\[
\xi_{\theta}(s):=
\begin{cases}
\begin{aligned}
    & -s, && s< -\theta,\\
    &-\frac{s^4}{8\theta^3}+3\frac{s^2}{4\theta}+\frac{3\theta}{8}, && -\theta \leq s\leq \theta,\\
    & s, && s> \theta,
\end{aligned}    
\end{cases}
\]
which satisfies $\xi_{\theta}\in C^2(\mathbb{R})$, $\xi_{\theta}(\cdot)\to |\cdot|$ uniformly as $\theta\to 0$, and
\[|s|\leq |\xi_{\theta}(s)|, \quad \left|\frac{d}{ds}\xi_{\theta}(s)\right|\leq C, \quad \left|\frac{d^2}{ds^2}\xi_{\theta}(s)\right|\leq \frac{C}{\theta}.\]

\noindent
We finally define
\begin{equation}
    \label{wp15}
    \chi_{\theta,\Delta_1}(\vec{x}-\vec{x}^*):=H_{\theta}\left(\Delta_1^2-|\vec{x}-\vec{x}^*|^2\right),\quad \chi_{\theta,\Delta_2}(c-c^*):=H_{\theta}\left(\Delta_2^2-|c-c^*|^2\right),
\end{equation}
which satisfy \eqref{wp12}-\eqref{wp14} by construction. We then introduce the following regularized form of the nonlocal term \eqref{wp5}
\begin{equation}
    \label{wp16}
\Sigma_{\theta}^{\delta}(f):=G^{\delta}\Sigma_{\theta}(f):=G^{\delta}(\mathbf{x})\int_0^1\int_{\Omega}{P}_{\theta,\Delta_1,\Delta_2}(\vec{x},\vec{x}^*,c,c^*)\left(\vec{x}-\vec{x}^*\right)f(\vec{x}^*,c^*,t)\,d\vec{x}^*dc^*,
\end{equation}
where the regularized bounded confidence kernel has the form:
\begin{align}
\label{wp17}
&{P}_{\theta,\Delta_1,\Delta_2}(\vec{x},\vec{x}^*,c,c^*):={\chi}_{\theta,\Delta_1}\left(\vec{x}-\vec{x}^*\right){\chi}_{\theta,\Delta_2}\left(c-c^*\right).
\end{align}
Thanks to the properties \eqref{wp12}, \eqref{wp14} and to the properties of $G^{\delta}(\cdot)$ \pier{it is} easy to prove the following properties \betti{of} the operator $\Sigma_{\theta}^{\delta}$ defined in \eqref{wp16}.

\begin{lemma}
\label{lem2}
For fixed $t\in [0,T]$, the operator $\Sigma_{\theta}^{\delta}$ is a bounded linear operator which maps $L^1(\Omega \times [0,1])$ in $W^{2,\infty}\left(\Omega \times [0,1];\mathbb{R}^d\right)$, and 
\begin{itemize}
    \item[(i)] \pier{for any $f(t)\in L^1(\Omega \times (0,1))$ it holds that}
    \begin{align*}
    &|\Sigma_{\theta}^{\delta}(f(t))|\leq \Delta_1\lVert f(t)\rVert_{L^1(\Omega \times \betti{(0,1)}}, \;\,  |\div \Sigma_{\theta}^{\delta}(f(t))|,\, \leq \frac{C}{\min(\theta,\delta)}\lVert f(t)\rVert_{L^1(\Omega \times \betti{(0,1)})}, \\
    & |\nabla \Sigma_{\theta}(f(t))|\leq \frac{C}{\min(\theta,\delta)}\lVert f(t)\rVert_{L^1(\Omega \times \betti{(0,1)})},\;\, |\nabla \div \Sigma_{\theta}(f)(t)|\leq \frac{C}{\min(\theta^2,\delta^2)}\lVert f(t)\rVert_{L^1(\Omega \times \betti{(0,1)})}.
    \end{align*}
\end{itemize}
\end{lemma}
\noindent
We finally introduce the following regularization $D_{\epsilon}(\cdot)$ of the mobility function $D(\cdot)$:
\begin{equation}
    \label{wp18}
    D_{\epsilon}(s):=D(s)+\sqrt{\epsilon}.
\end{equation}
Then, we introduce the following regularized version of problem \eqref{wp1}-\eqref{wp3}, which depends on the regularization parameters $\theta,\delta,\epsilon>0$: 

\noindent
\textbf{Problem} $\mathcal{P}^{\theta,\delta,\epsilon}$:
find a solution $f:\Omega\times [0,1]\times[0,T]\to \mathbb{R}$ to the equation
\begin{equation}
    \label{wp19}
    \partial_tf-\div \left(\Sigma_{\theta}^{\delta}(f)f\right)-D_{\epsilon}(c)\Delta_xf-\sqrt{\epsilon} \partial^2_cf=0,
\end{equation}
valid in $\Omega\times [0,1]\times[0,T]$, endowed with homogeneous Neumann boundary conditions
\begin{equation}
    \label{wp20}
    \begin{cases}
    \begin{aligned}
    &\Sigma_{\theta}^{\delta}(f)f\cdot \vec{n}+D_{\epsilon}(c)\nabla_xf\cdot \vec{n}=0 && \text{on} \; \partial \Omega\times [0,1]\times[0,T], \\ &\sqrt{\epsilon} \partial_cf\biggr|_{c\in\{0,1\}}=0 && \text{in} \; \Omega\times[0,T],
    \end{aligned}
    \end{cases}
\end{equation}
and with the initial condition
\begin{equation}
    \label{wp21}
    f(\vec{x},c,0)=f_0(\vec{x},c)\quad \forall \vec{x} \in \Omega, c\in [0,1],
\end{equation}
with 
$f_0\geq 0$, $\int_0^1\int_{\Omega}f_0d\vec{x}dc=1$.
\begin{remark}
    \label{rembc}
    Thanks to the property $(i)$ of Lemma~\ref{lem1d}, the boundary conditions \eqref{wp20} are equivalent to
    \begin{equation}
    \label{wp20d}
    \begin{cases}
    \begin{aligned}
    &\nabla_xf\cdot \vec{n}=0 && \text{on} \; \partial \Omega\times [0,1]\times[0,T], \\  &\partial_cf\biggr|_{c\in\{0,1\}}=0 && \text{in} \; \Omega\times[0,T].
    \end{aligned}
    \end{cases}
\end{equation}
\end{remark}

\section{Well posedness of $\mathcal{P}^{\theta,\delta,\epsilon}$}
\pier{Our intention is now to show} the following existence and regularity result for \eqref{wp19}-\eqref{wp21}.
\begin{theorem}
\label{thm1}
Let $\theta,\delta,\epsilon>0$, and assume that
\begin{itemize}
\item[(A1)] $f_0\in L^{2}(\Omega\times (0,1))$, with $f_0\geq 0$ a.e. in $\Omega \times [0,1]$ and $\int_0^1\int_{\Omega}f_0d\vec{x}dc=1$.
\end{itemize}
Then, there exists a unique weak solution $f$ of \eqref{wp19}-\eqref{wp21}, with
\begin{equation}
\label{wp21b}
f\in L^{\infty}(0,T;L^2(\Omega \times (0,1)))\cap L^{2}(0,T;H^1(\Omega \times [0,1]))\cap H^1\bigl(0,T;\left(H^1(\Omega \times [0,1])\right)'\bigr),
\end{equation}
which satisfies the weak formulation
\begin{equation}
    \label{wp22}\
\int_0^T\left<\partial_tf,g\right>\abramo{dt}+\int_0^T(\Sigma_{\theta}^{\delta}(f)f,\nabla_x g)\abramo{dt}+\int_0^T(D_{\epsilon}(c)\nabla_x f,\nabla_x g)
\abramo{dt}+\sqrt{\epsilon}\int_0^T\left(\partial_cf,\partial_cg\right)\abramo{dt}=0
\end{equation}
for all $g\in L^{2}(0,T;H^1(\Omega \times [0,1]))$, \pier{the initial condition} $f(\vec{x},c,0)=f_0(\vec{x},c)$ a.e. in $\Omega \times [0,1]$, and with the further properties that $f\geq 0$ a.e. in $\Omega \times [0,1]\times [0,T]$ and $\int_0^1\int_{\Omega}fd\vec{x}dc=1$. The weak solution enjoys the following continuous dependence property on the initial data: let $f_1^0,f_2^0\in L^{2}(\Omega\times (0,1))$ be two different initial data, and let $f_1,f_2$ be the two corresponding weak solutions, then the following estimate holds
\begin{equation}
    \label{wp23}
    \lVert f_1-f_2\rVert_{L^{\infty}(0,s;L^{2}(\Omega\times (0,1)))}\leq C\lVert f_1^0-f_2^0\rVert_{L^{2}(\Omega\times (0,1))},
\end{equation}
for any $s\in (0,T]$.
\pier{Next, assume that}
\begin{itemize}
\item[(A2)] $a=b=1$ or $a,b\geq 2$ \pier{in the mobility function  $D(c)$ defined by \eqref{wp4}};
\item[(A3)] \pier{$f_0\in H^{1}(\Omega\times [0,1])$, with $f_0\geq 0$ a.e. in $\Omega \times [0,1]$ and $\int_0^1\int_{\Omega}f_0d\vec{x}dc=1$.}
\end{itemize} 
Then, there exists a unique strong solution $f$ of \eqref{wp19}-\eqref{wp21}, with
\begin{equation}
\label{wp23a}
f\in L^{\infty}(0,T;H^1(\Omega \times (0,1)))\cap L^{2}(0,T;H^2(\Omega \times [0,1]))\cap H^1\bigl(0,T;L^2(\Omega \times [0,1])\bigr),
\end{equation}
which satisfies \eqref{wp19}-\eqref{wp21} \pier{almost everywhere.}
Moreover, the strong solution enjoys the following strong continuous dependence property on the initial data: let $f_1^0,f_2^0\in H^{1}(\Omega\times [0,1])$ be two different initial data, and let $f_1,f_2$ be the two corresponding weak solutions, then the following estimate holds
\begin{equation}
    \label{wp23s}
    \lVert f_1-f_2\rVert_{L^{\infty}(0,s;H^{1}(\Omega\times [0,1]))}\leq C\lVert f_1^0-f_2^0\rVert_{H^{1}(\Omega\times [0,1])},
\end{equation}
for any $s\in (0,T]$.
\end{theorem}
    In the following we will prove Theorem~\ref{thm1} through different steps, starting by formulating a time-discretization scheme and then studying its limit as the discretization parameter tends to zero. Given $N\in \mathbb{N}^+$, we consider a uniform partition of the interval $[0,T]$ into $N$ sub-intervals of width $\Delta t=\frac{T}{N}$, with $N+1$ extremal nodes $t^n=n\Delta t$, $n=0,\dots,N$. For a given $f:\Omega\times [0,1]\times [0,T]\to \mathbb{R}$, we introduce the notation
    \[
    f(\vec{x},c,t^n)=:f^n(\vec{x},c).
    \]
    We also introduce the function $\tilde{f}_0\in H^1(\Omega\times [0,1])$ which, given $f_0\in L^2(\Omega\times (0,1))$, is the unique solution of the variational problem
    \[
    (\tilde{f}_0,\chi)+\Delta t(\nabla \tilde{f}_0,\nabla \chi)=(f_0,\chi),
    \]
    for all $\chi\in H^1(\Omega\times [0,1])$. By a Lax--Milgram estimate we observe that
    \begin{equation}
        \label{wp23b}
        \lVert \nabla \tilde{f}_0\rVert^2\leq \frac{\lVert{f}_0\rVert^2}{2\Delta t}\leq \frac{C}{\Delta t}.
    \end{equation}
    Moreover, we have that $\tilde{f}_0\to {f}_0$ strongly in $L^2(\Omega\times (0,1))$ as $\Delta t\to 0$.
    Then, starting from the initial condition 
    \begin{equation}
    \label{wp23c}
    f^0:=
    \begin{cases}
    \tilde{f}_0 \quad \text{if} \quad f_0\in L^2(\Omega\times (0,1)),\\
    f_0 \quad \text{if} \quad f_0\in H^1(\Omega\times [0,1]),
    \end{cases}
    \end{equation}
    for $n\in \{0,\dots,N-1\}$, we define the following time-discrete approximation of \eqref{wp19}-\eqref{wp21}:

    \noindent
    \textbf{Problem} $\mathcal{P}_n^{\theta,\delta,\epsilon}$: given $f^n\in L^2(\Omega\times (0,1))$, with $f^n\geq 0$ a.e. in $\Omega\times [0,1]$ and $\int_0^1\int_{\Omega}f^n\,d\vec{x}dc=1$,
find a solution $f^{n+1}$ to the equation
\begin{equation}
    \label{wp24}
    \frac{f^{n+1}-f^n}{\Delta t}-\div \left(\Sigma_{\theta}^{\delta}(f^n)f^{n+1}\right)-D_{\epsilon}(c)\Delta_xf^{n+1}-\sqrt{\epsilon}  \partial^2_cf^{n+1}=0,
\end{equation}
valid in $\Omega\times [0,1]$, endowed with the boundary conditions
\begin{equation}
    \label{wp25}
    \begin{cases}
    \begin{aligned}
    &\Sigma_{\theta}^{\delta}(f^n)f^{n+1}\cdot \vec{n}+D_{\epsilon}(c)\nabla_xf^{n+1}\cdot \vec{n}=0 && \text{on} \; \partial \Omega\times [0,1], \\ 
    & \sqrt{\epsilon} \partial_c f^{n+1}\biggr|_{c\in\{0,1\}}=0 && \text{in} \; \Omega,
    \end{aligned}
    \end{cases}
\end{equation}
which, thanks to Remark~\ref{rembc}, are equivalent to 
\begin{equation}
    \label{wp25d}
    \begin{cases}
    \begin{aligned}
    & \nabla_xf^{n+1}\cdot \vec{n}=0 && \text{on} \; \partial \Omega\times [0,1], \\ 
    & \partial_c f^{n+1}\biggr|_{c\in\{0,1\}}=0 && \text{in} \; \Omega.
    \end{aligned}
    \end{cases}
\end{equation}
The following existence and regularity theorem is valid for \eqref{wp24}-\eqref{wp25d}.
\begin{theorem}
    \label{thm2}
     For any $n\in \{0,\dots,N-1\}$, given $\epsilon>0$ and \abramo{$f^n\in H^1(\Omega\times [0,1])$}, with $f^n\geq 0$ a.e. in $\Omega\times [0,1]$ and $\int_0^1\int_{\Omega}f^n\,d\vec{x}dc=1$, if $\Delta t<\frac{4\sqrt{\epsilon}}{\Delta_1^2}$ there exists a unique weak solution $f^{n+1}$ to \eqref{wp24}-\eqref{wp25} (uniformly in $\theta$ and $\delta$) such that $f^{n+1}\in H^1({\Omega}\times [0,1])$, with $f^{n+1}\geq 0$ in $\Omega\times [0,1]$ and $\int_0^1\int_{\Omega}f^{n+1}\,d\vec{x}dc=1$. Moreover, given $\epsilon,\delta>0$, $f^{n+1}\in H^2(\bar{\Omega}\times [0,1])$ (uniformly in $\theta$), i.e. it is a strong solution to \eqref{wp24}-\eqref{wp25d}.
\end{theorem}
\begin{proof}
The existence of a unique weak solution $f^{n+1}\in H^1(\Omega\times [0,1])$ to \eqref{wp24}-\eqref{wp25} for a sufficiently small value of $\Delta t$ is a consequence of the Lax--Milgram lemma. In particular, starting from \abramo{$f^n\in H^1(\Omega\times [0,1])$}, with $f^n\geq 0$ a.e. in $\Omega \times [0,1]$ and $\int_0^1\int_{\Omega}f^n\,d\vec{x}dc=1$, multiplying \eqref{wp24} by a function $\chi \in H^1(\Omega \times [0,1])$ and integrating over $\Omega \times [0,1]$ we obtain, after integration by parts and employing the boundary conditions \eqref{wp25}, that
\begin{equation}
\label{wp26}
\mathcal{A}_n(f^{n+1},\chi)=\mathcal{F}_n(\chi),
\end{equation}
where $\mathcal{A}_n: H^1(\Omega\times [0,1])\times H^1(\Omega\times [0,1])\to \mathbb{R}$ is the bilinear form associated to the weak formulation of \eqref{wp24} defined as
\[
\mathcal{A}_n(\psi,\chi):=\left(\frac{\psi}{\Delta t},\chi\right)+(\Sigma_{\theta}^{\delta}(f^n)\psi,\nabla_x \chi)+ (D_{\epsilon}(c)\nabla_x\psi,\nabla_x \chi)+\sqrt{\epsilon}\left(\partial_c \psi,\partial_c \chi\right),
\]
and $\mathcal{F}_n:H^1(\Omega\times [0,1])\to \mathbb{R}$ is defined as
\[
\mathcal{F}_n(\chi):=\left(\frac{f^{n}}{\Delta t},\chi\right).
\]
Due to the fact that $f^n\in H^1(\Omega\times [0,1])$, we observe that $\mathcal{F}_n$ is a linear functional over $H^1(\Omega\times [0,1])$. Moreover, the continuity of $\mathcal{A}_n$ is a consequence of property $(ii)$ 
in Lemma~\ref{lem1d}. Indeed, thanks to the facts that $f^n\geq 0$ a.e. in $\Omega \times [0,1]$ and that $\int_0^1\int_{\Omega}f^n\,d\vec{x}dc=1$, from property $(ii)$ in Lemma~\ref{lem1d} we have that $\lVert \Sigma_{\theta}^{\delta}(f^n)\rVert_{L^{\infty}(\Omega\times (0,1))}\leq \Delta_1$. Hence, \pier{it is} easy to prove that there exists $C>0$ independent on $\Delta t$ such that
\[
|\mathcal{A}_n(\psi,\chi)|\leq C \lVert \psi\rVert_{H^1(\Omega\times [0,1])}\lVert \chi\rVert_{H^1(\Omega\times [0,1])}, \quad \forall \psi,\chi\in H^1(\Omega\times [0,1]).
\]
Moreover, we have, for any $\pier{\rho}>0$, that
\begin{align*}
    &\mathcal{A}_n(\psi,\psi)\geq \frac{\lVert \psi\rVert^2}{\Delta t}+ \inf_{c\in [0,1]}D_{\epsilon}(c)\lVert\nabla_x \psi\rVert^2+\sqrt{\epsilon}\lVert \partial_c \psi\rVert^2-\lVert \Sigma_{\theta}^{\delta}(f^n)\Vert_{L^{\infty}(\Omega\times (0,1))}\int_0^1\int_{\Omega}|\psi|\,|\nabla_x \psi|\abramo{d\mathbf{x}dc}\\
    & \quad \geq (\sqrt{\epsilon}-\pier{\rho}  \Delta_1)\lVert \nabla \psi \rVert^2+\left(\frac{1}{\Delta t}-\frac{\Delta_1}{4\pier{\rho}}\right)\lVert\psi\rVert^2,
\end{align*}
hence, for $\Delta t<\frac{4\sqrt{\epsilon}}{\Delta_1^2}$, there exists $C>0$ depending on $\Delta t$ such that
\[
\mathcal{A}_n(\psi,\psi)\geq C\lVert \psi\rVert_{H^1(\Omega\times [0,1])}^{\pier{2}}\quad \forall \psi\in H^1(\Omega\times [0,1]).
\]
Then, by application of the Lax--Milgram lemma there exists a unique $f^{n+1}\in H^1(\Omega\times [0,1])$ solution of \eqref{wp26}. 
Taking $\chi\equiv 1$ in \eqref{wp26} we deduce that $\int_0^1\int_{\Omega}f^{n+1}\,d\vec{x}dc=\int_0^1\int_{\Omega}f^{n}\,d\vec{x}dc=1$.
Moreover, integrating by parts and employing the property $(i)$ 
in Lemma~\ref{lem1d}, we rewrite \eqref{wp26} as
\begin{equation}
\label{wp26b}
 (D_{\epsilon}(c)\nabla_xf^{n+1},\nabla_x \chi)+\sqrt{\epsilon}\left(\partial_c f^{n+1},\partial_c \chi\right)=(\mathcal{G}(f^{n+1},f^n),\chi),
\end{equation}
where
\[
\mathcal{G}(f^{n+1},f^n):=\frac{f^n-f^{n+1}}{\Delta t}+\div(\Sigma_{\theta}^{\delta}(f^n)f^{n+1}),
\]
and we observe that $(\mathcal{G}(f^{n+1},f^n),1)=0$. Also, as a consequence of the facts that $f^{n+1}\in H^1(\Omega \times [0,1])$ and \abramo{$f^n\in H^1(\Omega \times [0,1])$} and thanks to property $(i)$ of Lemma~\ref{lem2}, it holds that 
\[
\mathcal{G}(f^{n+1},f^n)\in L^2(\Omega\times (0,1)).
\]
Hence, by elliptic regularity, thanks to the smoothness assumptions on $\partial \Omega$, we have that
$f^{n+1}\in H^2(\bar{\Omega}\times [0,1])$.  
This means that the unique weak solution of \eqref{wp26} is actually a strong solution and satisfies \eqref{wp24} a.e. in $\Omega\times [0,1]$, with \eqref{wp25} satisfied in the sense of traces and almost everywhere. Note that this is valid only in the case $\epsilon,\delta>0$, while it is independent on the value of the parameter $\theta.$ \pier{Before proceeding in the proof, now we point out a remark.}
\begin{remark}
    \label{rem:h3}
    \abramo{Since $f^n\in H^1(\Omega \times [0,1])$ for any $n\in \{0,\dots,N-1\}$}, thanks to the fact that $f^{n+1}\in H^2(\bar{\Omega}\times [0,1])$, we have that \abramo{$\mathcal{G}(f^{n+1},f^n)\in H^1(\Omega\times [0,1])$}. Hence, a further application of elliptic regularity gives that $f^{n+1}\in H^3(\bar{\Omega}\times [0,1])$.
\end{remark}

We \pier{continue} by taking $\chi=\xi'_{\rho}(f^{n+1})$, for a given $\rho>0$, in \eqref{wp26}. We obtain that
\begin{align*}
    &(f^{n+1}-f^n,\xi'_{\rho}(f^{n+1}))+\Delta t (D_{\epsilon}(c)\xi''_{\rho}(f^{n+1})\nabla_x f^{n+1},\nabla_x f^{n+1})+\Delta t \sqrt{\epsilon}\left(\xi''_{\rho}(f^{n+1})\partial_c f^{n+1},\partial_c f^{n+1}\right)\\
    & \quad = \Delta t(\Sigma_{\theta}^{\delta}(f^n)f^{n+1},\xi''_{\rho}(f^{n+1})\nabla_x f^{n+1}).
\end{align*}
Using the Cauchy--Schwarz and the Young inequalities and the convexity of $\xi_{\rho}(\cdot)$ we obtain that
\begin{align*}
    & (\xi_{\rho}(f^{n+1}),1)+\Delta t \sqrt{\epsilon}(\xi''_{\rho}(f^{n+1})\nabla_x f^{n+1},\nabla_xf^{n+1})+\Delta t \sqrt{\epsilon}\left(\xi''_{\rho}(f^{n+1})\partial_c f^{n+1} ,\partial_c f^{n+1}\right)\\
    & \quad \leq (\xi_{\rho}(f^{n}),1)+\frac{\Delta t}{2} \sqrt{\epsilon}(\xi''_{\rho}(f^{n+1})\nabla_x f^{n+1},\nabla_xf^{n+1})
    \pier{{}+\frac{\Delta t}{2\sqrt{\epsilon}}\left(\xi''_{\rho}(f^{n+1})|f^{n+1}|^2,|\Sigma_{\theta}^{\delta}(f^{n})|^2\right)}.
\end{align*}
Using property $(ii)$ of Lemma~\ref{lem1d}, the hypotheses on $f^n$ and the definition of $\xi_{\rho}$ we have that
\begin{align}
\label{wp27}
    & \notag (\xi_{\rho}(f^{n+1}),1)+\frac{\Delta t}{2} \sqrt{\epsilon}(\xi''_{\rho}(f^{n+1})\nabla_x f^{n+1},\nabla_xf^{n+1})+ \pier{\Delta t}
    \sqrt{\epsilon}\left(\xi''_{\rho}(f^{n+1})\partial_c f^{n+1} ,\partial_c f^{n+1}\right)\\
    & \notag \quad \leq (\xi_{\rho}(f^{n}),1)+ \frac{\Delta t}{2\sqrt{\epsilon}}\lVert\Sigma_{\theta}^{\delta}(f^{n})\rVert_{L^{\infty}(\Omega\times (0,1);\mathbb{R}^d)}^2\left(\xi''_{\rho}(f^{n+1})|f^{n+1}|^2,1\right)\\
    & \notag \quad \leq (\xi_{\rho}(f^{n}),1)+ \frac{\Delta t\Delta_1^2}{2\sqrt{\epsilon}}\lVert f^n\rVert_{L^1(\Omega\times [0,1])}^2\int_0^1\int_{\Omega}\frac{3\left(\rho^2-(f^{n+1})^2\right)}{2\rho^3}|f^{n+1}|^2\mathbbm{1}_{\{|f^{n+1}|\leq \rho\}}\abramo{d\mathbf{x}dc}\\
    & \quad \leq (\xi_{\rho}(f^{n}),1) +
    \pier{\frac{C\Delta t}{\sqrt{\epsilon}} \rho
    \left(|f^n|, 1\right)^2,}
\end{align}
\pier{where the constant $C$ is independent of $\epsilon$ and $\rho$.}
Taking the limit in \eqref{wp27} as $\rho\to 0$ we finally obtain that
\begin{equation}
    \label{wp28}
    (|f^{n+1}|,1)\leq (|f^{n}|,1)=1.
\end{equation}
As a consequence, we observe that
\[
1=\int_0^1\int_{\Omega}f^{n+1}\abramo{d\mathbf{x}dc}\leq \int_0^1\int_{\Omega}|f^{n+1}|\abramo{d\mathbf{x}dc}\leq \int_0^1\int_{\Omega}|f^{n}|\abramo{d\mathbf{x}dc}=1,
\]
and hence 
\[
\int_0^1\int_{\Omega}f^{n+1}\abramo{d\mathbf{x}dc}=\int_0^1\int_{\Omega}|f^{n+1}|\abramo{d\mathbf{x}dc}=1.
\]
This implies that $f^{n+1}\geq 0$ a.e. in $\Omega\times [0,1]$, and due to the continuity property $f^{n+1}\in C^0(\bar{\Omega}\times [0,1])$ \pier{(cf.~Remark~\ref{rem:h3})} we \pier{deduce} that $f^{n+1}\geq 0$ for all $(\vec{x},c)\in \bar{\Omega}\times [0,1]$.

\noindent
Summarizing, we have the following
existence result for problem $\mathcal{P}_n^{\theta,\delta,\epsilon}$: given \abramo{$f^n\in H^1(\Omega\times [0,1])$}, with $f^n\geq 0$ a.e. in $\Omega\times [0,1]$ and $\int_0^1\int_{\Omega}f^n\,d\vec{x}dc=1$, if $\Delta t<\frac{4\sqrt{\epsilon}}{\Delta_1^2}$
there exists a unique weak solution $f^{n+1}$ to \eqref{wp24}-\eqref{wp25} such that $f^{n+1}\in H^1({\Omega}\times [0,1])$, with $f^{n+1}\geq 0$ in $\Omega\times [0,1]$ and $\int_0^1\int_{\Omega}f^{n+1}\,d\vec{x}dc=1$, which is a strong solution.
With the latter existence result, we can inductively proceed from $n=0$, considering that \abramo{$f^0\in H^1(\Omega\times [0,1])$,} 
with $f^0\geq 0$ a.e. in $\Omega\times [0,1]$ and $\int_0^1\int_{\Omega}f^0\,d\vec{x}dc=1$, to assert that, for any $n\in \{0,\dots,N-1\}$, if $\Delta t<\frac{4\sqrt{\epsilon}}{\Delta_1^2}$
there exists a unique weak solution $f^{n+1}$ to \eqref{wp24}-\eqref{wp25} such that $f^{n+1}\in H^1({\Omega}\times [0,1])$, with $f^{n+1}\geq 0$ in $\Omega\times [0,1]$ and $\int_0^1\int_{\Omega}f^{n+1}\,d\vec{x}dc=1$, which is a strong solution.
\end{proof}

The result \eqref{wp28} represents a first a-priori estimate which is independent \pier{of} $\Delta t$, $\theta$, and $\delta$, but depends on $\epsilon^{-1/2}$. In the following lemma we derive further a-priori estimates uniform in the parameter $\Delta t$ which will let us characterize the limit problem obtained as $\Delta t\to 0$. 
\begin{lemma}
    \label{lem3}
    The following a-priori estimates, which are uniform in the parameters $\Delta t$, $\theta$ and $\delta$, are valid for the weak solution of \eqref{wp24}-\eqref{wp25} under Assumption $(A1)$:
    \begin{equation}
    \label{wp30}
    \displaystyle \sup_{n\in\{1,\dots,N\}}\lVert f^{n}\rVert^2\leq C(T,\epsilon), \quad \Delta t\sum_{n=0}^{N-1}\lVert \nabla f^{n+1}\rVert^2\leq C(T,\epsilon),
\end{equation}
where the constants $C(T,\epsilon)\to +\infty$ as $\epsilon \to 0^+$. Moreover, under Assumption $(A3)$, the following a-priori estimates, which are uniform in the parameter $\Delta t$, are valid for the weak solution of \eqref{wp24}-\eqref{wp25}:
\begin{align}
    \label{wp30c}
    & \sum_{n=0}^{N-1}\lVert f^{n+1}-f^n\rVert^2\leq  C\left(T,\epsilon,\delta,\theta\right)\Delta t, \quad \sum_{n=0}^{N-1} \lVert \nabla (f^{n+1}-f^n)\rVert^2\leq  C\left(T,\epsilon,\delta,\theta\right),
\end{align}
where the constants $C(T,\epsilon,\delta,\theta)\to +\infty$ as $\epsilon,\delta,\theta \to 0^+$.
Finally, under Assumptions \pier{$(A2)$--$(A3)$} the following a-priori estimates, which are uniform in the parameter $\Delta t$, are valid for the strong solution of \eqref{wp24}-\eqref{wp25}:
\begin{align}
    \label{wp33}
    & \displaystyle \sup_{n\in\{1,\dots,N\}}\lVert \nabla f^{n}\rVert^2\leq  C(T,\epsilon,\delta,\theta), \quad \Delta t\sum_{n=0}^{N-1}\lVert \Delta f^{n+1}\rVert^2\leq  C(T,\epsilon,\delta,\theta), 
\end{align}
where the constants $C(T,\epsilon,\delta,\theta)\to +\infty$ as $\epsilon,\delta,\theta \to 0^+$.
\end{lemma}
\begin{proof}
We start by taking $\chi=f^{n+1}$ in \eqref{wp26}, obtaining, upon using the Cauchy--Schwarz and Young inequalities and property $(ii)$ of Lemma~\ref{lem1d}, that
\begin{align*}
    &\frac{\lVert f^{n+1}\rVert^2}{2}+\frac{\lVert f^{n+1}-f^n\Vert^2}{2}+\Delta t\sqrt{\epsilon}\left(\lVert \nabla_xf^{n+1}\rVert^2+\lVert \partial_c f^{n+1}\rVert^2\right)\leq \frac{\lVert f^{n}\rVert^2}{2}\\
    & \qquad +\Delta t \lVert \Sigma_{\theta}^{\delta}(f^n)\rVert_{L^{\infty}(\Omega \times (0,1);\mathbb{R}^d)}\lVert f^{n+1}\rVert \lVert \nabla_xf^{n+1}\rVert \leq \frac{\lVert f^{n}\rVert^2}{2}\\
    & \qquad+\frac{\Delta t}{2}\sqrt{\epsilon}\left(\lVert \nabla_xf^{n+1}\rVert^2+\lVert \partial_c f^{n+1}\rVert^2\right)+\frac{C\Delta t}{\sqrt{\epsilon}}\lVert f^{n+1}\rVert^2,
\end{align*}
from which, summing over $n\in\{0,\dots,N-1\}$, we obtain that
\begin{equation}
    \label{wp29}
    \frac{\lVert f^{N}\rVert^2}{2}+\sum_{n=0}^{N-1}\frac{\lVert f^{n+1}-f^n\Vert^2}{2}+\pier{\frac{\Delta t}2}\sqrt{\epsilon}\sum_{n=0}^{N-1}\lVert \nabla f^{n+1}\rVert^2\leq \frac{\lVert f^{0}\rVert^2}{2}+\frac{C\Delta t}{\sqrt{\epsilon}}\sum_{n=0}^{N-1}\lVert f^{n+1}\rVert^2.
\end{equation}
Then, by application of a discrete Gronwall inequality we have that
\begin{equation}
    \label{wp30b}
    \displaystyle \sup_{n\in\{1,\dots,N\}}\frac{\lVert f^{n}\rVert^2}{2}\leq \frac{\lVert f^{0}\rVert^2}{2}\mathrm{e}^{\frac{CT}{\sqrt{\epsilon}}}\leq C(T,\epsilon), \quad \Delta t\sum_{n=0}^{N-1}\lVert \nabla f^{n+1}\rVert^2\leq \frac{C}{\sqrt{\epsilon}}+\frac{CT}{\epsilon}\mathrm{e}^{\frac{CT}{\sqrt{\epsilon}}}\leq C(T,\epsilon),
\end{equation}
where the constants $C(T,\epsilon)\to +\infty$ as $\epsilon \to 0^+$.
In order to prove \eqref{wp30c}, we take $\chi=f^{n+1}-f^n$ in \eqref{wp26} and, upon integration by parts and employing the property $(i)$ of Lemma~\ref{lem1d}, using moreover the Cauchy--Schwarz inequality, we obtain that
\begin{align*}
    & \lVert f^{n+1}-f^n\rVert^2+\frac{\Delta t}{2}\left(D_{\epsilon}(c)\nabla_x f^{n+1},\nabla_x f^{n+1}\right)+\frac{\Delta t}{2}\left(D_{\epsilon}(c)\nabla_x (f^{n+1}-f^n),\nabla_x (f^{n+1}-f^n)\right)\\
        & \qquad +\frac{\sqrt{\epsilon}}{2}\Delta t\left(\partial_c f^{n+1},\partial_c f^{n+1}\right) + \frac{\sqrt{\epsilon}}{2}\Delta t\left(\partial_c(f^{n+1}-f^n),\partial_c(f^{n+1}-f^n)\right)\leq \frac{\Delta t}{2}\left(D_{\epsilon}(c)\nabla_x f^{n},\nabla_x f^{n}\right)\\
        & \qquad +\frac{\sqrt{\epsilon}}{2}\Delta t\left(\partial_c f^{n},\partial_c f^{n}\right)+\Delta t \lVert \div \Sigma_{\theta}^{\delta}(f^n)\rVert_{L^{\infty}(\Omega \times (0,1))}\lVert f^{n+1}\rVert \lVert f^{n+1}-f^n\rVert\\
        & \qquad +\Delta t \lVert \Sigma_{\theta}^{\delta}(f^n)\rVert_{L^{\infty}(\Omega \times (0,1);\mathbb{R}^d)}\lVert \nabla_xf^{n+1}\rVert \lVert f^{n+1}-f^n\rVert.
\end{align*}
Using now the Cauchy--Schwarz and Young inequalities, property $(i)$ 
of Lemma~\ref{lem2}, \eqref{wp28} and \eqref{wp30}, we have that
\begin{align}
\label{wp30d}
    & \notag \lVert f^{n+1}-f^n\rVert^2+\frac{\Delta t}{2}\left(D_{\epsilon}(c)\nabla_x f^{n+1},\nabla_x f^{n+1}\right)+\frac{\sqrt{\epsilon}}{2}\Delta t\left(\partial_c f^{n+1},\partial_c f^{n+1}\right)\\
    & \notag \qquad +\frac{\Delta t}{2}\sqrt{\epsilon}\left(\nabla (f^{n+1}-f^n),\nabla (f^{n+1}-f^n)\right)\leq \frac{\Delta t}{2}\left(D_{\epsilon}(c)\nabla_x f^{n},\nabla_x f^{n}\right)+\frac{\sqrt{\epsilon}}{2}\Delta t\left(\partial_c f^{n},\partial_c f^{n}\right)\\
    & \notag \qquad +\frac{C}{\min(\delta,\theta)}\Delta t\lVert f^{n+1}\rVert \lVert f^{n+1}-f^n\rVert+C\Delta t \lVert \nabla_xf^{n+1}\rVert \lVert f^{n+1}-f^n\rVert\\
    & \notag \quad \leq \frac{\Delta t}{2}\left(D_{\epsilon}(c)\nabla_x f^{n},\nabla_x f^{n}\right)+\frac{\sqrt{\epsilon}}{2}\Delta t\left(\partial_c f^{n},\partial_c f^{n}\right)+\frac{1}{2}\lVert f^{n+1}-f^n\rVert^2+\frac{C}{\min(\delta,\theta)^2}(\Delta t)^2\\
    & \qquad +C(\Delta t)^2\lVert \nabla f^{n+1}\rVert^2,
\end{align}
from which, summing over $n\in\{0,\dots,N-1\}$ and employing \eqref{wp30} and Assumption $(A3)$, we obtain that
\begin{equation*}
    \frac{1}{2}\sum_{n=0}^{N-1}\lVert f^{n+1}-f^n\rVert^2+\frac{\Delta t}{2}\sqrt{\epsilon}\sum_{n=0}^{N-1}\lVert \nabla f^{n+1}\rVert^2+\frac{\Delta t}{4}\sqrt{\epsilon}\sum_{n=0}^{N-1}\lVert \nabla (f^{n+1}-f^n)\rVert^2\leq C\left(T,\epsilon,\delta,\theta\right)\Delta t,
\end{equation*}
where $C\left(T,\epsilon,\delta,\theta\right)\to +\infty$ as $\epsilon,\delta,\theta\to 0$, which gives \eqref{wp30c}.
Considering now that $f^{n+1}\in H^2(\bar{\Omega}\times [0,1])$, we can obtain an higher order a-priori estimate by multiplying \eqref{wp24} by 
$$\pier{-\Delta f^{n+1}=-\Delta_x f^{n+1}-\frac{\partial^2 f^{n+1}}{\partial c^2}}$$
and integrating over $\Omega\times [0,1]$. Upon integration by parts and employing the boundary conditions \eqref{wp25d}, using moreover the Cauchy--Schwarz inequality, 
we obtain that
\begin{align*}
    &\frac{\lVert \nabla f^{n+1}\rVert^2}{2}+\frac{\lVert\nabla( f^{n+1}-f^n)\Vert^2}{2}+\Delta t\sqrt{\epsilon}\lVert \Delta f^{n+1}\rVert^2+\Delta t \left(D(c)\Delta_x f^{n+1},\Delta_x f^{n+1} \right)\\
    & \qquad +\Delta t \left(D(c)\Delta_x f^{n+1}, \frac{\partial^2 f^{n+1}}{\partial c^2} \right) \leq \frac{\lVert \nabla f^{n}\rVert^2}{2} +\Delta t \lVert \Sigma_{\theta}^{\delta}(f^n)\rVert_{L^{\infty}(\Omega \times (0,1);\mathbb{R}^d)}\lVert \nabla_x f^{n+1}\rVert \lVert \Delta f^{n+1}\rVert\\
    & \qquad +\Delta t \lVert \div\Sigma_{\theta}^{\delta}(f^n)\rVert_{L^{\infty}(\Omega \times (0,1))}\lVert f^{n+1}\rVert \lVert \Delta f^{n+1}\rVert,
\end{align*}
where we explicitly used the form of \eqref{wp18}. After integration by parts and employment of the boundary conditions \eqref{wp25d} and of the fact that $D(0)=D'(0)=D(1)=D'(1)=0$ (thanks to the hypothesis $(A2)$ on $a,b$), we get that
\begin{align}
\label{wp30cc}
    & \notag \int_0^1\int_{\Omega}D(c)\Delta_xf^{n+1}\partial_c^2 f^{n+1}\abramo{d\mathbf{x}dc}=-\int_0^1\int_{\Omega}D(c)\nabla_xf^{n+1}\cdot \partial_c^2 \nabla_x f^{n+1}\abramo{d\mathbf{x}dc}\\
    & \notag \quad =\int_0^1\int_{\Omega}D'(c)\nabla_xf^{n+1}\cdot \partial_c \nabla_x f^{n+1}\abramo{d\mathbf{x}dc} +\int_0^1\int_{\Omega}D(c)\nabla_x\left(\partial_c f^{n+1}\right)\cdot \nabla_x\left(\partial_c f^{n+1}\right)\abramo{d\mathbf{x}dc}\\
    & \quad =-\frac{1}{2}\int_0^1\int_{\Omega}D''(c)\left|\nabla_x f^{n+1}\right|^2\abramo{d\mathbf{x}dc}+\int_0^1\int_{\Omega}D(c)\nabla_x\left(\partial_c f^{n+1}\right)\cdot \nabla_x\left(\partial_c f^{n+1}\right)\abramo{d\mathbf{x}dc}.
\end{align}
\begin{remark}
\label{rem2}
\abramo{\pc{We note that the integral containing the third-order derivatives in the first line of \eqref{wp30cc}, as well as the associated boundary terms (which are not reported here since they vanish by virtue of the boundary conditions), is} well defined. Indeed, \abramo{\pc{as} we have regularized the initial condition in \eqref{wp23c} such that} $f^0\in H^1(\Omega\times [0,1])$, we have (see Remark~\ref{rem:h3}) that $f^{n+1}\in H^3(\bar{\Omega}\times [0,1])$, which makes the integrands integrable in the aforementioned integrals.}
\end{remark}
Hence, thanks to the boundedness of $D'(c)$ and $D''(c)$ (due to the hypothesis $(A2)$ on $a$ and $b$) and using the property $(i)$ of Lemma~\ref{lem2} and the Young inequality, we obtain that
\begin{align*}
    & \frac{\lVert \nabla f^{n+1}\rVert^2}{2}+\frac{\lVert\nabla( f^{n+1}-f^n)\Vert^2}{2}+\Delta t\sqrt{\epsilon}\lVert \Delta f^{n+1}\rVert^2+\Delta t \left(D(c)\Delta_x f^{n+1},\Delta_x f^{n+1} \right)\\
    & \qquad +\Delta t \left(D(c)\nabla_x \left(\partial_c f^{n+1}\right), \nabla_x \left(\partial_c f^{n+1}\right) \right) \leq \frac{\lVert \nabla f^{n}\rVert^2}{2} +\Delta t \lVert \Sigma_{\theta}^{\delta}(f^n)\rVert_{L^{\infty}(\Omega \times (0,1);\mathbb{R}^d)}\lVert \nabla_x f^{n+1}\rVert \lVert \Delta f^{n+1}\rVert\\
    & \qquad +\Delta t \lVert \div\Sigma_{\theta}^{\delta}(f^n)\rVert_{L^{\infty}(\Omega \times (0,1))}\lVert f^{n+1}\rVert \lVert \Delta f^{n+1}\rVert{{}+C\Delta t \lVert \nabla f^{n+1}\rVert^2}\leq \frac{\lVert \nabla f^{n}\rVert^2}{2}+\frac{\Delta t}{2}\sqrt{\epsilon}\lVert \Delta f^{n+1}\rVert^2\\
    & \qquad +\frac{C}{\sqrt{\epsilon}}\Delta t\lVert \nabla f^{n+1}\rVert^2+\frac{C}{\sqrt{\epsilon}\min(\delta,\theta)^2}\Delta t\lVert f^{n+1}\rVert^2 +C\Delta t\lVert \nabla f^{n+1}\rVert^2,
\end{align*}
from which we conclude, using also \eqref{wp30} (without reporting the dependence of the constants on the parameter $T$), that
\begin{align}
    \label{wp31}
    & \notag \frac{\lVert \nabla f^{n+1}\rVert^2}{2}+\frac{\lVert\nabla( f^{n+1}-f^n)\Vert^2}{2}+\frac{\Delta t}{2}\sqrt{\epsilon}\lVert \Delta f^{n+1}\rVert^2+\Delta t \left(D(c)\Delta_x f^{n+1},\Delta_x f^{n+1} \right)\\
    & \quad \leq \frac{\lVert \nabla f^{n}\rVert^2}{2}+\frac{C}{\sqrt{\epsilon}}\Delta t\lVert \nabla f^{n+1}\rVert^2+\frac{C(T,\epsilon)}{\sqrt{\epsilon}\min(\delta,\theta)^2}\Delta t.
\end{align}
Summing \eqref{wp31} over $n\in\{0,\dots,N-1\}$, we obtain that
\begin{align}
    \label{wp32}
    & \notag \frac{\lVert \nabla f^{N}\rVert^2}{2}+\sum_{n=0}^{N-1}\frac{\lVert \nabla(f^{n+1}-f^n)\Vert^2}{2}+\Delta t\sqrt{\epsilon}\sum_{n=0}^{N-1}\lVert \Delta f^{n+1}\rVert^2+\Delta t \sum_{n=0}^{N-1} \left(D(c)\Delta_x f^{n+1},\Delta_x f^{n+1} \right) \\ 
    & \quad \leq \frac{\lVert \nabla f^{0}\rVert^2}{2}+\frac{C\Delta t}{\sqrt{\epsilon}}\sum_{n=0}^{N-1}\lVert \nabla f^{n+1}\rVert^2+\frac{C(T,\epsilon)T}{\sqrt{\epsilon}\min(\delta,\theta)^2}.
\end{align}
Then, by application of a discrete Gronwall inequality and assuming that $f^0\in H^1(\Omega \times [0,1])$, we have that
\begin{align}
    \label{wp33b}
    & \notag \displaystyle \sup_{n\in\{1,\dots,N\}}\frac{\lVert \nabla f^{n}\rVert^2}{2}\leq \left(\frac{\lVert \nabla f^{0}\rVert^2}{2}+\frac{C(T,\epsilon)T}{\sqrt{\epsilon}\min(\delta,\theta)^2}\right)\mathrm{e}^{\frac{CT}{\epsilon\delta^2}}\leq C(T,\epsilon,\delta,\theta),\\
    & \Delta t\sum_{n=0}^{N-1}\lVert \Delta f^{n+1}\rVert^2\leq  C(T,\epsilon,\delta,\theta), \quad \Delta t \sum_{n=0}^{N-1} \left(D(c)\Delta_x f^{n+1},\Delta_x f^{n+1} \right)\leq  C(T,\epsilon,\delta,\theta),
\end{align}
where the constants $C(T,\epsilon,\delta,\theta)\to +\infty$ as $\epsilon,\delta,\theta \to 0^+$.
\end{proof}
We now associate to the sequence of time discrete solutions $f^n$ of \eqref{wp26}, for $n\in \{1, \dots, N\}$, the following piecewise constant and piecewise linear time interpolants over the interval $[0,T]$:
\begin{equation}
    \label{wp34}
    \bar{f}_{\Delta t}(t):=f^{n+1}, \quad \ubar{f}_{\Delta t}(t):=f^{n}, \quad \hat{f}_{\Delta t}(t):=\frac{t-t^n}{\Delta t}f^{n+1}+\frac{t^{n+1}-t}{\Delta t}f^{n},
\end{equation}
for any $t\in [t_{n},t_{n+1}]$, $n\in \{0,\dots, N-1\}$. The weak formulation \eqref{wp26} of \eqref{wp24}-\eqref{wp25} can then be reformulated as
\begin{equation}
    \label{wp35}
   \left(\partial_t \hat{f}_{\Delta t},\chi\right)+(\Sigma_{\theta}^{\delta}(\ubar{f}_{\Delta t})\bar{f}_{\Delta t},\nabla_x \chi)+ (D_{\epsilon}(c)\nabla_x\bar{f}_{\Delta t},\nabla_x \chi)+\sqrt{\epsilon}\left(\partial_c \bar{f}_{\Delta t},\partial_c \chi\right)=0,
\end{equation}
valid for any $\chi \in H^1(\Omega\times [0,1])$. The next lemma provides the convergence results, based on the estimates obtained in Lemma~\ref{lem3}, needed to study the limit of $\mathcal{P}_n^{\theta,\delta,\epsilon}$ as $\Delta t \to 0$.
\begin{lemma}
    \label{lem4}
    Under Assumption $(A1)$ of Theorem~\ref{thm1} there exists a limit function $f$, with regularity
    \[
    f\in L^{\infty}(0,T;L^2(\Omega \times (0,1)))\cap L^{2}(0,T;H^1(\Omega \times [0,1]))\cap H^1\pier{\bigl(}0,T;\bigl(H^1(\Omega \times [0,1])\bigr)'\pier{\bigr)},
    \]
    such that, up to subsequences of the solution of $\mathcal{P}_n^{\theta,\delta,\epsilon}$, which we still label by the index $\Delta t$, the following convergence results hold as $\Delta t \to 0$:
    \begin{align}
    \label{conv1} & \hat{f}_{\Delta t} \overset{\ast}{\rightharpoonup} f \quad \text{in} \quad L^{\infty}\left(0,T;L^2(\Omega \times (0,1))\right)\cap L^{2}\left(0,T;H^1(\Omega \times [0,1])\right),\\
    \label{conv2} & \partial_t\hat{f}_{\Delta t} {\rightharpoonup} \partial_tf \quad \text{in} \quad L^{2}\left(0,T;\left(H^1(\Omega \times [0,1])\right)'\right),\\
    \label{conv3} & \hat{f}_{\Delta t} {\rightarrow} f \quad \text{in} \quad L^{2}\left(0,T;(L^r(\Omega \times (0,1))\right)\;\; \text{and} \;\; \text{a.e. in} \; \; \Omega\times (0,1)\times (0,T),\\
    \label{conv4} & \bar{f}_{\Delta t},\ubar{f}_{\Delta t} \overset{\ast}{\rightharpoonup} f \quad \text{in} \quad L^{\infty}\left(0,T;L^2(\Omega \times (0,1))\right)\cap L^{2}\left(0,T;H^1(\Omega \times [0,1])\right),\\
    \label{conv5} & \bar{f}_{\Delta t},\ubar{f}_{\Delta t} {\rightarrow} f \quad \text{in} \quad L^{2}\left(0,T;(L^r(\Omega \times (0,1))\right)\;\; \text{and} \;\; \text{a.e. in} \; \; \Omega\times (0,1)\times (0,T),\\
    \label{conv6} & \Sigma_{\theta}^{\delta}\left(\ubar{f}_{\Delta t}\right) {\rightarrow} \Sigma_{\theta}^{\delta}\left(f\right) \quad \text{in} \quad L^{2}\left(0,T;(L^{\infty}(\Omega \times (0,1);\mathbb{R}^d)\right),
    \end{align}
    where $r\geq 1$ for $d=2$ and $r\in [1,6)$ for $d=3$. Moreover, under the Assumptions \pier{$(A2)$--$(A3)$} of Theorem~\ref{thm1} there exists a limit function $f$, with regularity
    \[
    f\in L^{\infty}(0,T;H^1(\Omega \times [0,1]))\cap L^{2}(0,T;H^2(\bar{\Omega} \times [0,1]))\cap H^1\bigl(0,T;L^2(\Omega \times (0,1))\bigr),
    \]
    such that, up to subsequences of the solution of $\mathcal{P}_n^{\theta,\delta,\epsilon}$, which we still label by the index $\Delta t$, the following convergence results hold as $\Delta t \to 0$:
    \begin{align}
    \label{conv7} & \hat{f}_{\Delta t} \overset{\ast}{\rightharpoonup} f \quad \text{in} \quad \pier{L^{\infty}\bigl(0,T;H^1(\Omega \times [0,1])\bigr)},\\
    \label{conv8} & \partial_t\hat{f}_{\Delta t} {\rightharpoonup} \partial_tf \quad \text{in} \quad L^{2}\bigl(0,T;L^2(\Omega \times (0,1))\bigr),\\
    \label{conv9} & \hat{f}_{\Delta t} {\rightarrow} f \quad \text{in} \quad C^{0}(0,T;(L^r(\Omega \times (0,1)))\;\; \text{and} \;\; \text{a.e. in} \; \; \Omega\times (0,1)\times (0,T),\\
    \label{conv10} & \bar{f}_{\Delta t},\ubar{f}_{\Delta t} \overset{\ast}{\rightharpoonup} f \quad \text{in} \quad \pier{L^{\infty}\bigl(0,T;H^1(\Omega \times [0,1])\bigr)},\\
    \label{conv10b} & \bar{f}_{\Delta t}{\rightharpoonup} f \quad \text{in} \quad  L^{2}\pier{\bigl(}0,T;H^2(\bar{\Omega} \times [0,1])\pier{\bigr)},\\
    \label{conv11} & \bar{f}_{\Delta t},\ubar{f}_{\Delta t} {\rightarrow} f \quad \text{in} \quad L^{\infty}
    \bigl(0,T;L^r(\Omega \times (0,1))\bigr)\;\; \text{and} \;\; \text{a.e. in} \; \; \Omega\times (0,1)\times (0,T),\\
    \label{conv12} & \Sigma_{\theta}^{\delta}\pier{\bigl(}\ubar{f}_{\Delta t}\pier{\bigr)} {\rightarrow} \Sigma_{\theta}^{\delta}\pier{\bigl(}f\pier{\bigr)} \quad \text{in} \quad L^{\infty}\pier{\bigl(}\Omega \times (0,1)\times (0,T);\mathbb{R}^d\pier{\bigr)},\\
    \label{conv13} & \div \Sigma_{\theta}^{\delta}\left(\ubar{f}_{\Delta t}\pier{\bigr)} {\rightarrow} \div \Sigma_{\theta}^{\delta}\left(f\right) \quad \text{in} \quad L^{\infty}\pier{\bigl(}0,T;L^r(\Omega \times (0,1))\right),
    \end{align}
    where $r\geq 1$ for $d=2$ and $r\in [1,6)$ for $d=3$.
\end{lemma}
\begin{remark}
The previous convergence results readily hold for $d=1$, yielding better compactness properties. 
This consideration remains valid for the limit studies in the upcoming section.
\end{remark}
\begin{proof}
We start by proving the convergence results \eqref{conv1}-\eqref{conv6}, which are valid under the Assumption $(A1)$.
As a consequence of Lemma~\eqref{lem3}, the piecewise time interpolants $\bar{f}_{\Delta t},\ubar{f}_{\Delta t}$ satisfy the following a-priori estimates, which are uniform in the parameters $\Delta t$, $\theta$ and $\delta$:
\begin{equation}
\label{ap1}
\lVert \bar{f}_{\Delta t}\rVert_{L^{\infty}\left(0,T;L^2(\Omega \times (0,1))\right)\cap L^{2}\left(0,T;H^1(\Omega \times [0,1])\right)}+\lVert \ubar{f}_{\Delta t}\rVert_{L^{\infty}\left(0,T;L^2(\Omega \times (0,1))\right)\cap L^{2}\left(0,T;H^1(\Omega \times [0,1])\right)}\leq C.
\end{equation}
The uniform bound in $L^{\infty}\left(0,T;L^2(\Omega \times (0,1))\right)$ is a direct consequence of the first estimate in \eqref{wp30}. Thanks to the second estimate of \eqref{wp30} we have that
\[
\int_0^T\lVert \nabla \bar{f}_{\Delta t}\rVert^2dt=\Delta t \sum_{n=0}^{N-1}\lVert \nabla f^{n+1}\rVert^{\pier{2}}\leq C.
\]
Moreover, thanks to the second estimate of \eqref{wp30} and to \eqref{wp23b} we have that
\[
\int_0^T\lVert \nabla \ubar{f}_{\Delta t}\rVert^2dt=\Delta t \sum_{n=0}^{N-1}\lVert \nabla f^{n}\rVert^{\pier{2}}\leq C.
\]
Hence, the Banach--Alaoglu theorem implies \pier{the convergence in~\eqref{conv4}, and we will explain in a while why the limit functions for $\bar{f}_{\Delta t}$ and $\ubar{f}_{\Delta t}$ are the same.} For what concerns the convergence properties of $\hat{f}_{\Delta t}$, we rewrite the third relation in \eqref{wp34} as
\[
\hat{f}_{\Delta t} \pier{(t)}=f^n+\frac{t-t^n}{\Delta t}(f^{n+1}-f^n).
\]
The following a-priori estimate, which is uniform in the parameters $\Delta t$ and $\theta$, is valid:
\begin{equation}
\label{ap2}
\lVert \hat{f}_{\Delta t}\rVert_{L^{\infty}\left(0,T;L^2(\Omega \times (0,1))\right)\cap L^{2}\left(0,T;H^1(\Omega \times [0,1])\right)}\leq C.
\end{equation}
\pier{In fact,} the uniform bound in $L^{\infty}\left(0,T;L^2(\Omega \times (0,1))\right)$ is a direct consequence of the first estimate in \eqref{wp30}. Thanks to the second estimate of \eqref{wp30} and to \eqref{wp23b} we have that
\begin{align*}
& \int_0^T\lVert \nabla \hat{f}_{\Delta t}\rVert^2dt\leq 2\Delta t \sum_{n=0}^{N-1}\lVert \nabla f^{n}\rVert^2+2\sum_{n=0}^{N-1}\int_{t^n}^{t^{n+1}}\frac{(t-t^n)^2}{\Delta t^2}\lVert \nabla (f^{n+1}-f^{n})\rVert^2\abramo{dt}\\
& \quad \leq 2\Delta t \sum_{n=0}^{N-1}\lVert \nabla f^{n}\rVert^2+\frac{2}{3}\Delta t \sum_{n=0}^{N-1}\lVert \nabla (f^{n+1}-f^n)\rVert^2\leq \frac{10}{3}\Delta t \sum_{n=0}^{N-1}\lVert \nabla f^n\rVert^2+\frac{4}{3}\Delta t \sum_{n=0}^{N-1}\lVert \nabla f^{n+1}\rVert^2\leq C.
\end{align*}
Multiplying \eqref{wp35} by a time function $\zeta \in L^2(0,T)$ and integrating in time over the interval $[0,T]$, we \pier{infer} that
\begin{align*}
   & \int_0^T \left(\partial_t \hat{f}_{\Delta t},\chi\right)\zeta\abramo{\,dt} \leq \sup_{c\in [0,1]}D_{\epsilon}(c)\int_0^T\lVert \nabla_x \bar{f}_{\Delta t} \rVert \, \lVert \nabla_x \chi \rVert\,|\zeta|\abramo{\,dt}+\sqrt{\epsilon}\int_0^T \lVert \partial_c \bar{f}_{\Delta t} \rVert \, \lVert \partial_c \chi \rVert\,|\zeta|\abramo{\,dt}\\
   & \qquad +C\lVert \ubar{f}_{\Delta t}\rVert_{L^{\infty}(0,T;L^1(\Omega \times (0,1)))} \int_0^T \lVert \bar{f}_{\Delta t} \rVert \, \lVert \nabla_x \chi \rVert\,|\zeta|\abramo{\,dt}\leq C \lVert \chi \rVert_{H^1(\Omega \times [0,1])}\lVert \zeta\rVert_{L^2(0,T)},
\end{align*}
where we used \eqref{ap1} and property $(ii)$ of Lemma~\ref{lem1d}, from which we \pier{obtain}~\eqref{conv2}. 
Hence, \eqref{conv1}-\eqref{conv2} together with the Aubin--Lions theorem give the compactness result \eqref{conv3}.
Note that the regularization of the initial condition introduced in \eqref{wp23c} is necessary to obtain the strong convergence of $\hat{f}_{\Delta t}$. We observe that the interpolants $\bar{f}_{\Delta t}$ and $\ubar{f}_{\Delta t}$ converge to the same limit as $\hat{f}_{\Delta t}$. This can be proved by showing that
\[
\lVert \hat{f}_{\Delta t}-\bar{f}_{\Delta t}\rVert_{L^2(0,T;L^2(\Omega\times (0,1)))}, \lVert \hat{f}_{\Delta t}-\ubar{f}_{\Delta t}\rVert_{L^2(0,T;L^2(\Omega\times (0,1)))}\to 0 \quad \text{as} \quad \Delta t \to 0. 
\]
Indeed, we can write
\[
\pier{\bigl(\hat{f}_{\Delta t}-\bar{f}_{\Delta t}\bigr) (t)}=\frac{t-t^{n+1}}{\Delta t}(f^{n+1}-f^n), \quad \pier{\bigl(\hat{f}_{\Delta t}-\ubar{f}_{\Delta t}\bigr)(t)}=\frac{t-t^{n}}{\Delta t}(f^{n+1}-f^n).
\]
Using the latter formula and the estimate \eqref{wp29} we conclude that
\[
\lVert \hat{f}_{\Delta t}-\bar{f}_{\Delta t}\rVert_{L^2(0,T;L^2(\Omega\times (0,1)))}^2, \lVert \hat{f}_{\Delta t}-\ubar{f}_{\Delta t}\rVert_{L^2(0,T;L^2(\Omega\times (0,1)))}^2\leq \frac{\Delta t}{3}\sum_{n=0}^{N-1}\lVert f^{n+1}-f^n\rVert^2\to 0\quad \text{as} \quad \Delta t \to 0.
\]
Hence, thanks to the strong convergence result \eqref{conv3} we also have that
\[
\bar{f}_{\Delta t},\ubar{f}_{\Delta t} {\rightarrow} f \quad \text{in} \quad \pier{L^{2}\bigl(0,T;L^2(\Omega \times (0,1))\bigr)}\;\; \text{and} \;\; \text{a.e. in} \; \; \Omega\times (0,1)\times (0,T).
\]
Due to the pointwise a.e. convergence of $\bar{f}_{\Delta t},\ubar{f}_{\Delta t}$ and the fact that they are uniformly (with respect to the parameters $\Delta t$, $\theta$ and $\delta$) bounded in $L^{2}\left(0,T;L^r(\Omega \times (0,1))\right)$, which is a consequence of \eqref{conv4}, we deduce \eqref{conv5}. \pier{Next, we show}~\eqref{conv6}, observing that
\[
\int_0^T\lVert \Sigma_{\theta}^{\delta}\left(\ubar{f}_{\Delta t}\right)-\Sigma_{\theta}^{\delta}\left(f\right)\rVert_{L^{\infty}(\Omega\times (0,1);\mathbb{R}^d)}^2\abramo{dt}\leq \Delta_1^2\int_0^T\lVert \ubar{f}_{\Delta t}-f\rVert_{L^{1}(\Omega\times (0,1))}^2\abramo{\,dt}\to 0,
\]
which is a consequence of property $(ii)$ of Lemma~\ref{lem1d}, of the linearity of the operator $\Sigma_{\theta}^{\delta}$ and of \eqref{conv5}.

We now come to the proof of \eqref{conv7}-\eqref{conv13}. The convergence results \eqref{conv7}, \eqref{conv10} and \eqref{conv10b} are a direct consequence of \eqref{wp33}, of Assumption $(A3)$ and of the Banach--Alaoglu theorem. Also, \eqref{conv8} \pier{follows from} the first estimate in~\eqref{wp30c}, since it implies that
\[
\int_0^T\lVert \partial_t \hat{f}_{\Delta t}\rVert^2\abramo{\,dt}=\Delta t \sum_{n=0}^{N-1}\frac{\lVert f^{n+1}-f^n\rVert^2}{(\Delta t)^2}\leq C.
\]
Hence, \eqref{conv7}, \eqref{conv8} and the Aubin--Lions theorem imply the compactness result \eqref{conv9}. With similar considerations as before, we calculate that
\[
\lVert \hat{f}_{\Delta t}-\bar{f}_{\Delta t}\rVert_{L^{\infty}(0,T;L^2(\Omega\times (0,1)))}^2, \lVert \hat{f}_{\Delta t}-\ubar{f}_{\Delta t}\rVert_{L^{\infty}(0,T;L^2(\Omega\times (0,1)))}^2\leq \sup_{n\in\{0,\dots,N-1\}}\lVert f^{n+1}-f^n\rVert^2\leq C\Delta t\to 0
\]
as $\Delta t\to 0$, where we used the first estimate of \eqref{wp30c}. Hence, thanks to the strong convergence result \eqref{conv9} we also have that
\[
\bar{f}_{\Delta t},\ubar{f}_{\Delta t} {\rightarrow} f \quad \text{in} \quad \pier{L^{\infty}\bigl(0,T;L^2(\Omega \times (0,1))\bigr)}\;\; \text{and} \;\; \text{a.e. in} \; \; \Omega\times (0,1)\times (0,T).
\]
Due to the pointwise a.e. convergence of $\bar{f}_{\Delta t},\ubar{f}_{\Delta t}$ and the fact that they are uniformly (with respect to the parameters $\Delta t$ and $\theta$) bounded in $\pier{L^{\infty}\bigl(0,T;L^r(\Omega \times (0,1))\bigr)}$, which is a consequence of \eqref{conv10}, we deduce \eqref{conv11}. We conclude by proving \eqref{conv12} and \eqref{conv13}. We observe that
\[
\lVert \Sigma_{\theta}^{\delta}\left(\ubar{f}_{\Delta t}\right)-\Sigma_{\theta}^{\delta}\left(f\right)\rVert_{L^{\infty}(\Omega\times (0,1)\times (0,T);\mathbb{R}^d)}\leq \Delta_1\lVert \ubar{f}_{\Delta t}-f\rVert_{L^{\infty}(0,T;L^{1}(\Omega\times (0,1)))}\to 0,
\]
which is a consequence of property $(ii)$ of Lemma~\ref{lem1d}, of the linearity of the operator $\Sigma_{\theta}^{\delta}$ and of \eqref{conv11}. Finally, in a similar way property $(i)$ of Lemma~\ref{lem2}, the linearity of $\div \Sigma_{\theta}^{\delta}$ and \eqref{conv11} imply \eqref{conv13}. 
\end{proof}
We are now ready to prove that the limit function $f$ obtained in Lemma~\ref{lem4} is the unique weak solution, under Assumption $(A1)$, or strong solution, under Assumptions \pier{$(A2)$--$(A3)$}, of \eqref{wp19}-\eqref{wp21}, as defined in 
Theorem~\ref{thm1}. This constitutes the proof of Theorem~\ref{thm1}.
\begin{proof}[Proof of Theorem~\ref{thm1}] We consider first the situation under Assumption $(A1)$. Let's multiply \eqref{wp35} by a generic function $\zeta\in L^2(0,T)$ and integrate it in time over the interval $[0,T]$. Defining $g:=\chi \zeta \in L^2(0,T;H^1(\Omega \times [0,1]))$, we obtain the weak formulation
 \begin{align}
    \label{wp36}
   & \int_0^T\left(\partial_t \hat{f}_{\Delta t},g\right)\abramo{\,dt}+\int_0^T(\Sigma_{\theta}^{\delta}(\ubar{f}_{\Delta t})\bar{f}_{\Delta t},\nabla_x g)\abramo{\,dt}\notag\\
   &\qquad{}+ \int_0^T(D_{\epsilon}(c)\nabla_x\bar{f}_{\Delta t},\nabla_x g)\abramo{\,dt}+\sqrt{\epsilon}\int_0^T\left(\partial_c \bar{f}_{\Delta t},\partial_c g\right)\abramo{\,dt}=0,
\end{align}
valid for any $g \in L^2(0,T;H^1(\Omega \times [0,1]))$. 
Thanks to the convergence results \eqref{conv2} and \eqref{conv4} we straightforwardly \pier{infer} that
\begin{align*}
    & \int_0^T\bigl(\partial_t \hat{f}_{\Delta t},g
    \bigr)\abramo{\,dt}\rightarrow \int_0^T\left<\partial_t f,g\right>\abramo{\,dt},\quad  \int_0^T(D_{\epsilon}(c)\nabla_x\bar{f}_{\Delta t},\nabla_x g)\abramo{\,dt} \rightarrow \int_0^T(D_{\epsilon}(c)\nabla_x f,\nabla_x g)\abramo{\,dt},\\
    &  \sqrt{\epsilon}\int_0^T\left(\partial_c \bar{f}_{\Delta t},\partial_c g\right)\abramo{\,dt} \rightarrow \sqrt{\epsilon}\int_0^T\left(\partial_c f,\partial_c g\right)\abramo{\,dt},
\end{align*}
as $\Delta t\to 0$. For what concerns the second term in \eqref{wp36}, we rewrite it as
\begin{align}
\label{wp36dt}
    & \notag \int_0^T(\Sigma_{\theta}^{\delta}(\ubar{f}_{\Delta t})\bar{f}_{\Delta t},\nabla_x g)\abramo{\,dt}= \underbrace{\int_0^T\left((\Sigma_{\theta}^{\delta}(\ubar{f}_{\Delta t})-\Sigma_{\theta}^{\delta}(f))\bar{f}_{\Delta t},\nabla_x g\right)\abramo{\,dt}}_{=:I_1}\\
    & \qquad {}+\underbrace{\int_0^T(\Sigma_{\theta}^{\delta}(f)(\bar{f}_{\Delta t}-f),\nabla_x g)\abramo{\,dt}}_{=:I_2}{}
    +\int_0^T(\Sigma_{\theta}^{\delta}(f)f,\nabla_x g)\abramo{\,dt}.
\end{align}
We observe that
\begin{align*}
&|I_1|\leq \lVert \Sigma_{\theta}^{\delta}(\ubar{f}_{\Delta t})-\Sigma_{\theta}^{\delta}(f)\rVert_{L^{2}(0,T;L^{\infty}(\Omega\times (0,1);\mathbb{R}^d))}\lVert \bar{f}_{\Delta t} \rVert_{L^{\infty}(0,T;L^2(\Omega\times (0,1)))}\lVert \nabla_x g\rVert_{L^2(0,T;L^2(\Omega\times (0,1)))}\\
& \quad \leq C\lVert \ubar{f}_{\Delta t}-f\rVert_{L^{2}(0,T;L^{1}(\Omega\times (0,1)))}\lVert \bar{f}_{\Delta t} \rVert_{L^{\infty}(0,T;L^2(\Omega\times (0,1)))}\lVert \nabla_x g\rVert_{L^2(0,T;L^2(\Omega\times (0,1)))}\rightarrow 0
\end{align*}
as $\Delta t \to 0$, thanks to \eqref{ap1} and \eqref{conv5}. Also, we have that
\begin{align*}
&|I_2|\leq \lVert \Sigma_{\theta}^{\delta}(f)\rVert_{L^{\infty}(0,T;L^{\infty}(\Omega\times (0,1);\mathbb{R}^d))}\lVert \bar{f}_{\Delta t}-f \rVert_{L^{2}(0,T;L^2(\Omega\times (0,1)))}\lVert \nabla_x g\rVert_{L^2(0,T;L^2(\Omega\times (0,1)))}\\
& \quad \leq C\lVert f\rVert_{L^{\infty}(0,T;L^{1}(\Omega\times (0,1)))}\lVert \bar{f}_{\Delta t}-f \rVert_{L^{2}(0,T;L^2(\Omega\times (0,1)))}\lVert \nabla_x g\rVert_{L^2(0,T;L^2(\Omega\times (0,1)))}\rightarrow 0
\end{align*}
as $\Delta t \to 0$, thanks to \eqref{wp21b} and \eqref{conv5}. We thus conclude that
\[
\int_0^T(\Sigma_{\theta}^{\delta}(\ubar{f}_{\Delta t})\bar{f}_{\Delta t},\nabla_x g)\abramo{\,dt}\rightarrow \int_0^T(\Sigma_{\theta}^{\delta}(f)f,\nabla_x g)\abramo{\,dt}
\]
as $\Delta t \to 0$. Hence, the limit function has the regularity \eqref{wp21b}, thanks to \eqref{conv1} and \eqref{conv2}, and satisfies the weak formulation \eqref{wp22}. We need finally to prove that the initial condition holds. Due to \eqref{wp21b}, in particular to the fact that $L^{2}(0,T;H^1(\Omega \times [0,1]))\cap H^1\pier{\bigl(}0,T;\left(H^1(\Omega \times [0,1])\right)'\pier{\bigr)}$ is continuously embedded in $C^{0}(0,T;L^2(\Omega \times (0,1)))$, and thanks to the fact that $\tilde{f}_0\to f_0$ strongly in $L^2(\Omega \times (0,1))$, \pier{it is} easy to prove in a standard way that $f(\cdot,\cdot,0)=f_0$ a.e. in $\Omega\times [0,1]$. Thanks to Fatou's lemma applied to \eqref{wp28}, we have that
\[
\int_0^1\int_{\Omega}|f|\abramo{\,d\mathbf{x}dc}\leq \int_0^1\int_{\Omega}f_0\abramo{\,d\mathbf{x}dc}=1.
\]
Moreover, taking $g\equiv 1$ in \eqref{wp22}, we have that $\int_0^1\int_{\Omega}f=1$. Hence, with similar arguments as before, we conclude that $f\geq 0$ a.e. in $\Omega\times [0,1]$. Coming to the uniqueness of $f$ as a solution of \eqref{wp22} and to its continuous dependence from data, let us assume that $f_1^0,f_2^0\in L^{2}(\Omega\times (0,1))$ be two different initial data, and let $f_1,f_2$ be the two corresponding weak solutions. We define $\tilde{f}:=f_1-f_2$. Taking the difference between \eqref{wp22} for $f_1$ and $f_2$, we have that $\tilde{f}$ satisfies the weak formulation
\begin{align}
    \label{wp37}
    & \notag \int_0^T\left<\partial_t \tilde{f},g\right>\abramo{\,dt}+\int_0^T(\Sigma_{\theta}^{\delta}(f_1)\tilde{f},\nabla_x g)\abramo{\,dt}+\int_0^T((\Sigma_{\theta}^{\delta}(f_1)-\Sigma_{\theta}^{\delta}(f_2))f_2,\nabla_x g)\abramo{\,dt}\\
    & \qquad +\int_0^T(D_{\epsilon}(c)\nabla_x \tilde{f},\nabla_x g)\abramo{\,dt}+\sqrt{\epsilon}\int_0^T\left(\partial_c \tilde{f},\partial_c g\right)\abramo{\,dt}=0
\end{align}
for all $g\in L^{2}(0,T;H^1(\Omega \times [0,1]))$, and moreover satisfies the initial condition $\tilde{f}(\cdot,\cdot,0)=f_1^0-f_2^0$ a.e. in $\Omega \times [0,1]$.
If we take $g(\vec{x},c,t)=\tilde{f}(\vec{x},c,t)H(s-t)$, for a given $s\in (0,T]$, in \eqref{wp37}, we obtain, using \eqref{wp21b}, property $(i)$ of Lemma~\ref{lem1}, the Cauchy--Schwarz and Young inequalities that
\begin{align*}
   & \frac{\betti{\lVert \tilde{f}\rVert^2}(s)}{2}+\sqrt{\epsilon}\int_0^s\lVert \nabla \tilde{f}\rVert^2\abramo{\,dt}\leq \frac{\betti{\lVert \tilde{f}\rVert^2(0)}}{2}+\lVert \Sigma_{\theta}^{\delta}(f_1)\rVert_{L^{\infty}(0,s;L^{\infty}(\Omega \times (0,1)))}\int_0^s\lVert \tilde{f}\rVert \, \lVert \nabla_x \tilde{f}\rVert\abramo{\,dt}\\
   & \qquad{}+\int_0^s\lVert\Sigma_{\theta}^{\delta}(f_1)-\Sigma_{\theta}^{\delta}(f_2)\rVert_{L^{\infty}(\Omega \times (0,1))}\lVert f_2\rVert\,\lVert \nabla_x \tilde{f}\rVert\abramo{\,dt}\leq \frac{\lVert \tilde{f}\rVert^{\pier{2}}(0)}{2}\\
   & \qquad{} +C\lVert f_1\rVert_{L^{\infty}(0,s;L^{1}(\Omega \times (0,1)))}\int_0^s\lVert \tilde{f}\rVert \, \lVert \nabla_x \tilde{f}\rVert\abramo{\,dt}+C\lVert f_2\rVert_{L^{\infty}(0,s;L^{2}(\Omega \times (0,1)))}\int_0^s\lVert \tilde{f}\rVert_{L^1(\Omega\times (0,1))} \, \lVert \nabla_x \tilde{f}\rVert\abramo{\,dt}\\
   & \quad\leq \frac{\betti{\lVert \tilde{f}\rVert^2}(0)}{2}+\frac{\sqrt{\epsilon}}{2}\int_0^s\lVert \betti{\nabla_x} \tilde{f}\rVert^2\abramo{\,dt}+\frac{C}{\sqrt{\epsilon}}\int_0^s\lVert \tilde{f}\rVert^2\abramo{\,dt}\abramo{\leq \frac{{\lVert \tilde{f}\rVert^2}(0)}{2}+\frac{\sqrt{\epsilon}}{2}\int_0^s\lVert {\nabla} \tilde{f}\rVert^2\abramo{\,dt}+\frac{C}{\sqrt{\epsilon}}\int_0^s\lVert \tilde{f}\rVert}^2\abramo{\,dt}.
\end{align*}
An application of the Gronwall inequality to the previous estimate leads to \eqref{wp23}, which also implies the uniqueness of the weak solution $f$. We have thus proved the first part of Theorem~\ref{thm1}.

\pier{We now consider the case in which Assumptions $(A2)$--$(A3)$ hold. Since, in this setting, $f^{n+1}\in H^2(\bar{\Omega}\times[0,1])$ for every $n\in\{0,\ldots,N-1\}$ and satisfies \eqref{wp25} in the trace sense, an integration by parts allows us to rewrite \eqref{wp35} as}
\begin{equation}
    \label{wp38}
   \left(\partial_t \hat{f}_{\Delta t},\chi\right)-\left(\div \left(\Sigma_{\theta}^{\delta}(\ubar{f}_{\Delta t})\bar{f}_{\Delta t}\right), \chi\right)- (D_{\epsilon}(c)\Delta_x\bar{f}_{\Delta t}, \chi)-\sqrt{\epsilon}\left(\partial_c^2 \bar{f}_{\Delta t},\chi\right)=0,
\end{equation}
valid for any $\chi \in L^2(\Omega\times (0,1))$. Multiplying \eqref{wp38} by a function $\zeta \in L^2(0,T)$ and integrating in time over the interval $[0,T]$, defining $g:=\chi \zeta$, we arrive at the weak formulation
\begin{align}
    \label{wp39}
   & \int_0^T\left(\partial_t \hat{f}_{\Delta t},g\right)\abramo{\,dt}-\int_0^T\left(\div \left(\Sigma_{\theta}^{\delta}(\ubar{f}_{\Delta t})\bar{f}_{\Delta t}\right), g\right)\abramo{\,dt}\nonumber\\
   &\qquad {}- \int_0^T(D_{\epsilon}(c)\Delta_x\bar{f}_{\Delta t}, g)\abramo{\,dt} -\sqrt{\epsilon}\int_0^T\left(\partial_c^2 \bar{f}_{\Delta t},g\right)\abramo{\,dt}=0,
\end{align}
valid for any $g \in L^2(0,T;L^2(\Omega\times (0,1)))$. We observe that the limit function has the regularity \eqref{wp23a}, thanks to \eqref{conv7}, \eqref{conv8} and \eqref{conv10b}. \pier{By virtue of} the convergence properties \eqref{conv8}, \eqref{conv10} and \eqref{conv10b} we directly have that
\begin{align*}
    & \int_0^T\left(\partial_t \hat{f}_{\Delta t},g\right)\abramo{\,dt} \rightarrow \int_0^T\left(\partial_t f,g\right)\abramo{\,dt}, \quad \int_0^T(D_{\epsilon}(c)\Delta_x\bar{f}_{\Delta t}, g)\abramo{\,dt}\rightarrow \int_0^T(D_{\epsilon}(c)\Delta_xf, g)\abramo{\,dt},\\
    &  \sqrt{\epsilon}\int_0^T\left(\partial_c^2 \bar{f}_{\Delta t},g\right)\abramo{\,dt}\rightarrow \sqrt{\epsilon}\int_0^T\left(\partial_c^2 f,g\right)\abramo{\,dt},
\end{align*}
as $\Delta t \to 0$.
For what concerns the second term in \eqref{wp38}, we rewrite it as
\begin{align*}
    &\int_0^T\left(\div \left(\Sigma_{\theta}^{\delta}(\ubar{f}_{\Delta t})\bar{f}_{\Delta t}\right), g\right)\abramo{\,dt}=\underbrace{\int_0^T\left(\left(\Sigma_{\theta}^{\delta}(\ubar{f}_{\Delta t})-\Sigma_{\theta}^{\delta}(f)\right)\cdot \nabla_x\bar{f}_{\Delta t}, g\right)\abramo{\,dt}}_{=:I_1}+\underbrace{\int_0^T\left(\Sigma_{\theta}^{\delta}(f)\cdot \nabla_x(\bar{f}_{\Delta t}-f), g\right)\abramo{\,dt}}_{=:I_2}\\
    & \qquad + \underbrace{\int_0^T\left(\div \left(\Sigma_{\theta}^{\delta}(\ubar{f}_{\Delta t})-\Sigma_{\theta}^{\delta}(f)\right)\bar{f}_{\Delta t}, g\right)\abramo{\,dt}}_{=:I_3}+\underbrace{\int_0^T\left(\div (\Sigma_{\theta}^{\delta}(f))(\bar{f}_{\Delta t}-f), g\right)\abramo{\,dt}}_{=:I_4}\\
    & \qquad +\int_0^T\left(\div \left(\Sigma_{\theta}^{\delta}(f)f\right), g\right)\abramo{\,dt}.
\end{align*}
We observe that
\begin{align*}
&|I_1|\leq \lVert \Sigma_{\theta}^{\delta}(\ubar{f}_{\Delta t})-\Sigma_{\theta}^{\delta}(f)\rVert_{L^{2}(0,T;L^{\infty}(\Omega\times (0,1);\mathbb{R}^d))}\lVert \nabla_x\bar{f}_{\Delta t} \rVert_{L^{\infty}(0,T;L^2(\Omega\times (0,1)))}\lVert g\rVert_{L^2(0,T;L^2(\Omega\times (0,1)))}\\
& \quad \leq C\lVert \ubar{f}_{\Delta t}-f\rVert_{L^{2}(0,T;L^{1}(\Omega\times (0,1)))}\lVert \betti{\nabla_x} \bar{f}_{\Delta t} \rVert_{L^{\infty}(0,T;L^2(\Omega\times (0,1)))}\lVert g\rVert_{L^2(0,T;L^2(\Omega\times (0,1)))}\rightarrow 0
\end{align*}
as $\Delta t \to 0$, thanks to \eqref{wp33} and \eqref{conv11}. For what concerns $I_2$, since $\Sigma_{\theta}^{\delta}(f)g\in L^2(0,T;L^2(\Omega\times (0,1)))$ and thanks to the weak convergence \eqref{conv10}, we have that
\[
I_2=\int_0^T\left(\nabla_x(\bar{f}_{\Delta t}-f), \Sigma_{\theta}(f)g\right)\abramo{\,dt}\rightarrow 0
\]
as $\Delta t \to 0$. As for $I_3$, we have that
\[
|I_3|\leq \lVert \div \left(\Sigma_{\theta}^{\delta}(\ubar{f}_{\Delta t})-\Sigma_{\theta}^{\delta}(f)\right)\rVert_{L^{\infty}(0,T;L^2(\Omega\times (0,1)))}\lVert \bar{f}_{\Delta t}\rVert_{L^{2}(0,T;L^{\infty}(\Omega\times (0,1)))}\lVert g\rVert_{L^{2}(0,T;L^2(\Omega\times (0,1)))}\rightarrow 0
\]
as $\Delta t \to 0$, thanks to \eqref{conv13}, \eqref{wp33} and to the Sobolev embedding $H^2(\bar{\Omega}\times [0,1])\subset\subset C^0(\bar{\Omega}\times [0,1])$. Finally, property $(i)$ of Lemma~\ref{lem2} and the regularity of $f$ imply that
\[
\div (\Sigma_{\theta}^{\delta}(f))g\in L^2\bigl(0,T;L^{2}(\Omega\times (0,1))\bigr).
\]
Hence, the convergence result \eqref{conv11} implies that
\[
I_4=\int_0^T\left(\bar{f}_{\Delta t}-f, \div (\Sigma_{\theta}^{\delta}(f))g\right)\abramo{\,dt}\rightarrow 0
\]
as $\Delta t \to 0$.
We finally conclude that
\[
\int_0^T\left(\div \left(\Sigma_{\theta}^{\delta}(\ubar{f}_{\Delta t})\bar{f}_{\Delta t}\right), g\right)\abramo{\,dt}\rightarrow \int_0^T\left(\div \left(\Sigma_{\theta}^{\delta}(f)f\right), g\right)\abramo{\,dt}
\]
as $\Delta t \to 0$. Hence, the limit function satisfies the strong formulation \eqref{wp19}-\eqref{wp20} a.e. in $\Omega\times [0,1]\times [0,T]$. For what concerns the attainment of the initial condition, the strong convergence \eqref{conv9} and the fact that $f_0\in L^2(\Omega \times (0,1))$ implies that $f(\cdot,\cdot,0)=f_0$ a.e. in $\Omega\times [0,1]$.

{\pier{We are left to establish the strong continuous dependence on the initial data stated in \eqref{wp23s}. Let $f_1^0,f_2^0\in H^{1}(\Omega\times[0,1])$ be two distinct initial data, and let $f_1$ and $f_2$ denote the corresponding strong solutions. We set \( \tilde{f}:=f_1-f_2. \) Subtracting equation \eqref{wp19} satisfied by $f_2$ from that satisfied by $f_1$, testing the resulting equation with $-\Delta\tilde{f}$, and integrating over $\Omega\times[0,1]\times[0,s]$, with $s\in(0,T]$, which is justified by the regularity property \eqref{wp23a}, we obtain, after integration by parts, that}
\begin{align*}
   & \frac{\lVert \nabla \tilde{f}\rVert\abramo{^2}(s)}{2}+\sqrt{\epsilon}\int_0^s\lVert \Delta \tilde{f}\rVert^2\abramo{\,dt}+\int_0^s\left(D(c)\Delta_x \tilde{f},\Delta\tilde{f}\right)\abramo{\,dt}
   = \frac{\lVert \nabla\tilde{f}\rVert\abramo{^2}(0)}{2}-\int_0^s\left(\div \Sigma_{\theta}^{\delta}(f_1) \tilde{f},\Delta \tilde{f}\right)\abramo{\,dt}\\
   & \qquad -\int_0^s\left(\Sigma_{\theta}^{\delta}(f_1)\cdot \nabla_x \tilde{f},\Delta \tilde{f}\right)\abramo{\,dt}-\int_0^s\left(\div \Sigma_{\theta}^{\delta}(\tilde{f}) f_2,\Delta \tilde{f}\right)\abramo{\,dt}-\int_0^s\left(\Sigma_{\theta}^{\delta}(\tilde{f})\cdot \nabla_x f_2,\Delta \tilde{f}\right)\abramo{\,dt}.
\end{align*}
Proceeding similarly to \eqref{wp30cc} to rewrite the third term on the left hand side, we obtain that
\begin{align*}
   & \frac{\lVert \nabla \tilde{f}\rVert\abramo{^2}(s)}{2}+\sqrt{\epsilon}\int_0^s\lVert \Delta \tilde{f}\rVert^2\abramo{\,dt}+\int_0^s\abramo{\int_0^1\int_{\Omega}}D(c)\left(\Delta_x \tilde{f}\right)^2\abramo{\,d\mathbf{x}dcdt}+\int_0^s\abramo{\int_0^1\int_{\Omega}}D(c)|\nabla_x(\partial_c\tilde{f})|^2\abramo{\,d\mathbf{x}dcdt}
   \\
   & \quad = \frac{\lVert \nabla\tilde{f}\rVert\abramo{^2}(0)}{2}+\frac{1}{2}\int_0^s\abramo{\int_0^1\int_{\Omega}}D''(c)|\nabla_x \tilde{f}|^2\abramo{\,d\mathbf{x}dcdt}-\int_0^s\left(\div \Sigma_{\theta}^{\delta}(f_1) \tilde{f},\Delta \tilde{f}\right)\abramo{\,dt}\\
   & \qquad -\int_0^s\left(\Sigma_{\theta}^{\delta}(f_1)\cdot \nabla_x \tilde{f},\Delta \tilde{f}\right)\abramo{\,dt}-\int_0^s\left(\div \Sigma_{\theta}^{\delta}(\tilde{f}) f_2,\Delta \tilde{f}\right)\abramo{\,dt}-\int_0^s\left(\Sigma_{\theta}^{\delta}(\tilde{f})\cdot \nabla_x f_2,\Delta \tilde{f}\right)\abramo{\,dt}.
\end{align*}
}
Using the Cauchy--Schwarz and Young inequalities, the properties $(ii)$ and $(iv)$ of Lemma~\ref{lem1d}, \eqref{wp23a} and \eqref{wp23} we \pier{infer} that
\begin{align*}
   & \frac{\lVert \nabla \tilde{f}\rVert^2(s)}{2}+\sqrt{\epsilon}\int_0^s\lVert \Delta \tilde{f}\rVert^2\abramo{\,dt}+\int_0^s\abramo{\int_0^1\int_{\Omega}}D(c)\left(\Delta_x \tilde{f}\right)^2\abramo{\,d\mathbf{x}dcdt}+\int_0^s\abramo{\int_0^1\int_{\Omega}}D(c)|\nabla_x(\partial_c\tilde{f})|^2\abramo{\,d\mathbf{x}dcdt}
   \\
   & \quad \leq \frac{\lVert \nabla\tilde{f}\rVert^2(0)}{2}+\frac{\sqrt{\epsilon}}{2}\int_0^s\lVert \Delta \tilde{f}\rVert^2\abramo{\,dt}+C\int_0^s\lVert \nabla \tilde{f}\rVert^2\abramo{\,dt}+\frac{C}{\sqrt{\epsilon}}\int_0^s\lVert \div \Sigma_{\theta}^{\delta}(f_1)\rVert_{L^{\infty}(\Omega\times (0,1))}^2\lVert \tilde{f}\rVert^2\abramo{\,dt}\\
   & \qquad +\frac{C}{\sqrt{\epsilon}}\int_0^s\lVert \div \Sigma_{\theta}^{\delta}(\tilde{f})\rVert_{L^{\infty}(\Omega\times (0,1))}^2\lVert f_2\rVert^2\abramo{\,dt}+\frac{C}{\sqrt{\epsilon}}\int_0^s\lVert \Sigma_{\theta}^{\delta}(f_1)\rVert_{L^{\infty}(\Omega\times (0,1);\mathbb{R}^d)}^2\lVert \nabla \tilde{f}\rVert^2\abramo{\,dt}\\
   & \qquad +\frac{C}{\sqrt{\epsilon}}\int_0^s\lVert \Sigma_{\theta}^{\delta}(\tilde{f})\rVert_{L^{\infty}(\Omega\times (0,1);\mathbb{R}^d)}^2\lVert \nabla f_2\rVert^2\abramo{\,dt}\leq \frac{\lVert \nabla\tilde{f}\rVert^2(0)}{2}+\frac{\sqrt{\epsilon}}{2}\int_0^s\lVert \Delta \tilde{f}\rVert^2\abramo{\,dt}\\
   & \qquad + C\int_0^s\lVert \nabla \tilde{f}\rVert^2\abramo{\,dt}+ \frac{C}{\sqrt{\epsilon}\delta^2}\int_0^s\lVert f_1\rVert_{W^{1,1}(\Omega\times (0,1))}^2\lVert \tilde{f}\rVert^2\abramo{\,dt}+\frac{C}{\sqrt{\epsilon}\delta^2}\int_0^s\lVert \tilde{f}\rVert_{W^{1,1}(\Omega\times (0,1))}^2\abramo{\,dt}\\
   & \quad \leq \frac{\lVert \nabla\tilde{f}\rVert^2(0)}{2}+\frac{C}{\sqrt{\epsilon}}\lVert \tilde{f}\rVert^2(0)+\frac{\sqrt{\epsilon}}{2}\int_0^s\lVert \Delta \tilde{f}\rVert^2\abramo{\,dt}+\frac{C}{\sqrt{\epsilon}\delta^2}\int_0^s\lVert \nabla \tilde{f}\rVert^2\abramo{\,dt}.
\end{align*}
An application of the Gronwall inequality to the previous estimate leads to \eqref{wp23s}.
\end{proof}

We conclude this section by proving a further regularity result for the weak solution of problem $\mathcal{P}^{\theta,\delta,\epsilon}$, which \abramo{is uniform in the parameters $\theta,\delta$ and} is valid under the assumption
\begin{itemize}
    \item[(A1bis)] $f_0\in L^{\infty}(\Omega\times (0,1))$, with $f_0\geq 0$ a.e. in $\Omega \times [0,1]$ and $\int_0^1\int_{\Omega}f_0d\vec{x}dc=1$.
\end{itemize}
As discussed in Remark~\ref{rem1}, this is the most natural regularity class required for the initial conditions to represent gray-scale images.
\begin{theorem}
    \label{thm3}
    Let $\theta,\delta,\epsilon>0$, and assume that Assumption $(A1bis)$ holds. Then, the weak solution of Theorem~\ref{thm1} enjoys the further regularity
    \begin{equation}
        \label{wp40}
    f\in L^{\infty}(0,T;L^{\infty}(\Omega\times (0,1))).
    \end{equation}
\end{theorem}
\begin{proof}
The proof is based on a Moser--Alikakos iteration argument. We start by deriving formal estimates which give the boundedness of the weak solution of 
Theorem~\ref{thm1} in $L^{\infty}(0,T;L^p(\Omega \times (0,1)))$, for any $p\in [1,+\infty)$. Note that the bound for $p\in [1,2]$ has already been obtained in \eqref{wp21b}. We proceed by formally taking $g=pf^{p-1}$, with $p\in (2,+\infty)$, in \eqref{wp22}. We note that, given the regularity \eqref{wp21b}, $pf^{p-1}$ is not 
an admissible test function in \eqref{wp22}. 
Rigorous estimates could be obtained by
considering a proper truncation of the test function of the form $g=p\max(R,f)^{p-1}$, for some $R>0$, and then taking the limit as $R\to +\infty$. Since these steps are standard, we show here the formal calculations only for the sake of brevity.

Observing that
\[
<\partial_tf,pf^{p-1}>=\frac{d}{dt}(f^p,1),
\]
which could be justified rigorously by means of truncations of the test function, observing that $f\geq 0$ a.e. in $\Omega\times (0,1)$ and adopting the same arguments as in \cite[Lemma~$2.5$]{grun}, we \pier{infer} that
\begin{align*}
    &(f^p,1)(T)+p(p-1)\int_0^T\left(D(c)\nabla_xf,f^{p-2}\nabla_xf\right)\abramo{\,dt}+\sqrt{\epsilon}p(p-1)\int_0^T\left(\partial_cf,f^{p-2}\partial_cf\right)\abramo{\,dt}\\
    & \quad =(f_0^p,1)-p(p-1)\int_0^T(\Sigma_{\theta}^{\delta}(f)f^{p-1},\nabla_xf)\abramo{\,dt}.
\end{align*}
Rewriting some terms, using moreover \eqref{wp18}, property $(ii)$ of 
Lemma~\ref{lem1d}, the Cauchy--Schwarz and the Young inequalities and \eqref{wp21b}, we obtain that
\begin{align*}
    &(f^p,1)(T)+4\sqrt{\epsilon}\frac{(p-1)}{p}\int_0^T\lVert \nabla \left(f^{p/2}\right) \rVert^2\abramo{\,dt}\leq (f_0^p,1)+C(p-1)\int_0^T\lVert f^{p/2}\rVert\,\lVert \nabla \left(f^{p/2}\right) \rVert\abramo{\,dt}\\
    & \quad \leq (f_0^p,1)+2\sqrt{\epsilon}\frac{(p-1)}{p}\int_0^T\lVert \nabla \left(f^{p/2}\right) \rVert^2\abramo{\,dt}+C\frac{p(p-1)}{\sqrt{\epsilon}}\int_0^T(f^p,1)\abramo{\,dt},
\end{align*}
from which we conclude that
\begin{equation}
    \label{wp40b}
    (f^p,1)(T)+2\sqrt{\epsilon}\frac{(p-1)}{p}\int_0^T\lVert \nabla \left(f^{p/2}\right) \rVert^2\abramo{\,dt}\leq (f_0^p,1)+C_1\frac{p(p-1)}{\sqrt{\epsilon}}\int_0^T(f^p,1)\abramo{\,dt}.
\end{equation}
A Gronwall argument then implies that $f\in L^{\infty}(0,T;L^p(\Omega\times (0,1)))$ for any $p\in [1,+\infty)$.

We continue from \eqref{wp40b}, rewriting the second term on the right hand side, employing a trilinear H\"{o}lder inequality with exponents $(6,3/2,6)$ and the Sobolev embdedding $H^1(\Omega\times [0,1])\subset L^6(\Omega\times (0,1))$, as
\begin{align*}
 & C_1\frac{p(p-1)}{\sqrt{\epsilon}}\int_0^T(f^p,1)\abramo{\,dt}\\
 & \quad =C_1\frac{p(p-1)}{\sqrt{\epsilon}}\int_0^T(f^{p/2},f^{p/2})\abramo{\,dt}\leq C_1|\Omega|^{1/6}\frac{p(p-1)}{\sqrt{\epsilon}}\int_0^T\left(\int_0^1\int_{\Omega}f^{\frac{3}{4}p}\abramo{\,d\mathbf{x}dc}\right)^{2/3}\lVert f^{p/2}\rVert_{L^6(\Omega\times (0,1))}\abramo{\,dt}\\
 & \quad \leq C_1C_2|\Omega|^{1/6}\frac{p(p-1)}{\sqrt{\epsilon}}\int_0^T\left[\int_0^1\int_{\Omega}\left(f^p+\left|\nabla\left(f^{p/2}\right)\right|^2\right)\abramo{\,d\mathbf{x}dc}\right]^{1/2}\left(\int_0^1\int_{\Omega}f^{\frac{3}{4}p}\right)^{2/3}\abramo{\,dt}\\
 & \quad \leq C_1\frac{p(p-1)}{2\sqrt{\epsilon}}\int_0^T(f^p,1)\abramo{\,dt}+C_1C_2^2|\Omega|^{1/3}\frac{p(p-1)}{2\sqrt{\epsilon}}\int_0^T\left(\int_0^1\int_{\Omega}f^{\frac{3}{4}p}\abramo{\,d\mathbf{x}dc}\right)^{4/3}\abramo{\,dt}\\
 & \qquad + \sqrt{\epsilon}\frac{(p-1)}{p}\int_0^T\lVert \nabla \left(f^{p/2}\right) \rVert^2\abramo{\,dt}+C_1^2C_2^2|\Omega|^{1/3}\frac{p^3(p-1)}{4\epsilon^{3/2}}\int_0^T\left(\int_0^1\int_{\Omega}f^{\frac{3}{4}p}\abramo{\,d\mathbf{x}dc}\right)^{4/3}\abramo{\,dt},
\end{align*}
where $C_2$ is the constant of the Sobolev immersion. Rearranging terms in the previous inequality and observing that $p\geq 1$, we obtain that
\begin{equation}
    \label{wp40c}
    C_1\frac{p(p-1)}{\sqrt{\epsilon}}\int_0^T(f^p,1)\abramo{\,dt}\leq 2\sqrt{\epsilon}\frac{(p-1)}{p}\int_0^T\lVert \nabla \left(f^{p/2}\right) \rVert^2\abramo{\,dt}+C\frac{p^4}{\epsilon^{3/2}}\int_0^T\left(\int_0^1\int_{\Omega}f^{\frac{3}{4}p}\abramo{\,d\mathbf{x}dc}\right)^{4/3}\abramo{\,dt}.
\end{equation}
Inserting \eqref{wp40c} in \eqref{wp40b} we finally obtain that
\begin{equation}
    \label{wp40d}
    (f^p,1)(T)\leq (f_0^p,1)+C\frac{p^4}{\epsilon^{3/2}}\int_0^T\left(\int_0^1\int_{\Omega}f^{\frac{3}{4}p}\abramo{\,d\mathbf{x}dc}\right)^{4/3}\abramo{\,dt}.
\end{equation}
Observing that \eqref{wp40d} is also valid for any $t\in(0,T]$, introducing the sequence $(l_k)_{k\geq 0}$ of real numbers defined by
\[
l_0=1,\;\; l_{k+1}=\frac{4}{3}l_k, \;\; k\in \mathbb{N},
\]
with $p=l_{k+1}$, and defining moreover the sequence $(\gamma_k)_{k\geq 0}$ of positive real numbers as
\[
\gamma_{k+1}:=\sup_{[0,T]}\left[\int_0^1\int_{\Omega}f^{l_{k+1}}\abramo{\,d\mathbf{x}dc}\right],
\]
we can rewrite \eqref{wp40d} as
\begin{equation}
    \label{wp40e}
    \gamma_{k+1}\leq C\left(\epsilon\right)l_{k+1}^4\max\left \{ \lVert f_0\rVert_{L^{\infty}(\Omega \times (0,1))}^{l_{k+1}}, \gamma_k^{4/3}\right \},
\end{equation}
with $C\left(\epsilon\right)\to +\infty$
as $\epsilon \to 0^+$.
Thanks to Lemma~$A.1$ in \cite{laurencot}, \eqref{wp40e} implies that $\left(\gamma_k^{1/l_k}\right)_{k\geq 0}$ is bounded, and \eqref{wp40} follows from taking the limit as $k\to +\infty$.
\end{proof}

\section{Passage to the limit in regularized systems
}

\pier{We start a new section which contain our limit procedures.}

\subsection{Limit problem $\mathcal{P}^{\theta,\delta,\epsilon}\to \mathcal{P}^{\epsilon}$ as $\theta,\delta\to 0$}
\pier{We first} consider the limit problem of $\mathcal{P}^{\theta,\delta,\epsilon}$ as $\theta,\delta\to 0$, giving an existence and regularity result for the case of the un-regularized nonlocal operator, keeping the diffusive regularizations in the system. Since the existence result of a unique weak solution of Theorem~\ref{thm2} and the a-priori estimate \eqref{wp30} of Lemma~\ref{lem3} are uniform in the parameters $\theta,\delta$, \pier{it is} easy to generalize the existence result of a weak solution of Theorem~\ref{thm1} to the case $\theta,\delta\to 0$. \pier{It is} sufficient to consider $\theta=\delta=\Delta t$ along the proof 
of Theorem~\ref{thm1}, taking the limit as $\Delta t \to 0$ in problem $\mathcal{P}_n^{\theta,\delta,\epsilon}$ and using the convergence properties \eqref{conv1}-\eqref{conv6} of Lemma~\ref{lem4}. 
\abramo{The only new part with respect to the proof of Theorem~\ref{thm1} is to consider the limit as $\theta,\delta\to 0$ in the second term in \eqref{wp36}, which entails to add and subtract the term $\int_0^T(\Sigma(f)f,\nabla_xg)\abramo{\,dt}$ in the right hand side of \eqref{wp36} and show that
\begin{equation}
\label{wp40td}
\int_0^T\left(\left(\Sigma_{\theta}^{\theta}(f)-\Sigma(f)\right)f,\nabla_xg\right)\abramo{\,dt}\to 0
\end{equation}
as $\theta\to 0$. \pier{It is} possible to prove, thanks to property \betti{(jv)} of $G^{\theta}$ and to \eqref{wp13}, that $\Sigma_{\theta}^{\theta}(f)-\Sigma(f)\to 0$ a.e. in $\Omega \times (0,1) \times (0,T)$. This will be shown later in the proof of Theorem~\ref{thm8}; \pier{therefore, we omit the} details here. Also, thanks to property $(i)$ of Lemma~\ref{lem1} and to the regularity of $f$\pier{,} the term $\Sigma_{\theta}^{\theta}(f)-\Sigma(f)$ is uniformly bounded. Then, since $f \nabla_xg\in L^1(\Omega \times (0,1)\times (0,T);\mathbb{R}^d)$, an application of the Dominated Convergence Theorem proves \eqref{wp40td}.}
We thus have the following existence theorem, which we report without proof since \pier{it is} straightforward.
\begin{theorem}
\label{thm4}
Let $\epsilon>0$, and \pier{let the Assumption~\pier{$(A1)$} of Theorem~\ref{thm1} holds.}
Then, there exists a unique weak solution $f$ of $\mathcal{P}^{0,0,\epsilon}$, with
\begin{equation}
\label{wp41}
f\in L^{\infty}(0,T;L^2(\Omega \times (0,1)))\cap L^{2}(0,T;H^1(\Omega \times [0,1]))\cap H^1\bigl(0,T;\left(H^1(\Omega \times [0,1])\right)'\bigr),
\end{equation}
which satisfies the weak formulation
\begin{equation}
    \label{wp42}
    \int_0^T\left<\partial_t f,g\right>\abramo{\,dt}+\int_0^T(\Sigma(f)f,\nabla_x g)\abramo{\,dt}+\int_0^T(D_{\epsilon}(c)\nabla_x f,\nabla_x g)\abramo{\,dt}+\sqrt{\epsilon}\int_0^T\left(\partial_c f,\partial_c g\right)\abramo{\,dt}=0
\end{equation}
for all $g\in L^{2}(0,T;H^1(\Omega \times [0,1]))$, satisfying the initial condition $f(\vec{x},c,0)=f_0(\vec{x},c)$ a.e. in $\Omega \times [0,1]$ and with the further properties that $f\geq 0$ a.e. in $\Omega \times [0,1]\times [0,T]$ and $\int_0^1\int_{\Omega}fd\vec{x}dc=1$. The weak solution enjoys the following continuous dependence property from the initial data: let $f_1^0,f_2^0\in L^{2}(\Omega\times (0,1))$ be two different initial data, and let $f_1,f_2$ be the two corresponding weak solutions, then the following estimate holds
\begin{equation}
    \label{wp43}
    \lVert f_1-f_2\rVert_{L^{\infty}(0,s;L^{2}(\Omega\times (0,1)))}\leq C\lVert f_1^0-f_2^0\rVert_{L^{2}(\Omega\times (0,1))},
\end{equation}
for any $s\in (0,T]$. 
\end{theorem}
Then, since also the a-priori estimate of Theorem~\ref{thm3} is uniform in $\theta$ and $\delta$, \pier{it is} possible to generalize Theorem~\ref{thm3} to the case $\theta=\delta=0$, by taking in the limit system~\eqref{wp42} the same testing procedure as in the proof of Theorem~\ref{thm3}. We thus have the following theorem. 
\begin{theorem}
    \label{thm5}
    Let $\epsilon>0$, and assume that Assumption $(A1bis)$ holds. Then, the weak solution of Theorem~\ref{thm4} enjoys the further regularity
    \begin{equation}
        \pier{\label{wp40new}}
    f\in L^{\infty}(0,T;L^{\infty}(\Omega\times (0,1))).
    \end{equation}
\end{theorem}

\subsection{Limit problem $\mathcal{P}^{\theta,\delta,\epsilon}\to \mathcal{P}^{\theta,\delta}$ as $\epsilon\to 0$}
The study of the limit problem of $\mathcal{P}^{\theta,\delta,\epsilon}$ as $\epsilon\to 0$ is more delicate than the previous one, since the a-priori estimates of Lemma~\ref{lem3} depend on inverse powers of $\epsilon$. We need to obtain a-priori estimates for the discrete solution of Problem $\mathcal{P}_n^{\theta,\delta,\epsilon}$ which are uniform in $\epsilon$, for given $\theta,\delta>0$, employing in particular the properties of the operator $\Sigma_{\theta}^{\delta}$ introduced in Lemma~\ref{lem2}. 
We are able to prove the following existence and regularity result. 
\begin{theorem}
\label{thm6}
Let $\theta,\delta>0$, \betti{and \pier{let the} Assumptions \pier{$(A2)$--$(A3)$} hold true}.
Then, there exists a unique strong solution $f$ of problem $\mathcal{P}^{\theta,\delta,0}$, with
\begin{equation}
\label{wp41new}
f\in L^{\infty}(0,T;H^1(\Omega \times (0,1)))\cap H^1\bigl(0,T;L^2(\Omega \times (0,1))\bigr), \quad \sqrt{D(c)}\Delta_x f\in L^2\bigl(0,T;L^2(\Omega \times [0,1])\bigr),
\end{equation}
which satisfies \eqref{wp19}-\eqref{wp21} (with $\epsilon=0$) a.e. in $\Omega\times [0,1]\times [0,T]$,
with $f(\vec{x},c,0)=f_0(\vec{x},c)$ a.e. in $\Omega \times [0,1]$ and with the further properties that $f\geq 0$ a.e. in $\Omega \times [0,1]\times [0,T]$ and $\int_0^1\int_{\Omega}fd\vec{x}dc=1$. Moreover, the strong solution enjoys the following continuous dependence property on the initial data: let $f_1^0,f_2^0\in H^{1}(\Omega\times [0,1])$ be two different initial data, and let $f_1,f_2$ be the two corresponding solutions, then the following estimate holds
\begin{equation}
    \label{wp43s2}
    \lVert f_1-f_2\rVert_{L^{\infty}(0,s;L^{2}(\Omega\times (0,1)))}\leq C\lVert f_1^0-f_2^0\rVert_{L^{2}(\Omega\times (0,1))},
\end{equation}
for any $s\in (0,T]$.
\end{theorem}
In order to prove Theorem~\ref{thm6}, we derive the following lemma.
\begin{lemma}
    \label{lem5}
    The following a-priori estimates, which are uniform in the parameters $\Delta t$ and $\epsilon$, are valid for the weak solution of \eqref{wp24}-\eqref{wp25} under Assumptions \pier{$(A2)$--$(A3)$}:
    \begin{align}
    \label{wp42new}
    & \notag \displaystyle \sup_{n\in\{1,\dots,N\}}(|f^{n}|,1)\leq (\xi_{\epsilon}(f^{0}),1)+C\sqrt{\epsilon},\quad \sup_{n\in\{1,\dots,N\}}\lVert f^{n}\rVert_{H^1(\Omega \times [0,1])}^2\leq C(T,\theta,\delta), \\
    & \notag\sqrt{\epsilon}\Delta t\sum_{n=0}^{N-1}\lVert \Delta f^{n+1}\rVert^2\leq C(T,\theta,\delta), \quad \Delta t\sum_{n=0}^{N-1}\left(D(c)\Delta_x f^{n+1},\Delta_x f^{n+1}\right)\leq C(T,\theta,\delta),\\
    & \sum_{n=0}^{N-1}\lVert f^{n+1}-f^n\rVert^2\leq  C(\theta,\delta)\Delta t,
\end{align}
where the constants $C(T,\theta,\delta), C(\theta,\delta)\to +\infty$ as $\theta,\delta\to 0^+$. 
\end{lemma}
\begin{proof}
We already know from Theorem~\ref{thm2} that, for any $\theta,\delta,\epsilon>0$ and for any $n\in\{0,\dots,N-1\}$, $f^n\geq 0$ a.e. in $\Omega\times [0,1]$ and $\int_0^1\int_{\Omega}f^n\,d\vec{x}dc=1$.
As a first step, we take $\chi=\xi'_{\epsilon}(f^{n+1})$,  \pier{with the parameter $\epsilon$ in $\xi_{\epsilon}$ being the same as the coefficient $\epsilon$ in equation~\eqref{wp24} and boundary conditions~\eqref{wp25}. Then, according 
to~\eqref{wp27} where now $\rho$ is replaced by $\epsilon$, we find out that} 
\begin{align*}
    & \notag (\xi_{\epsilon}(f^{n+1}),1)\leq (\xi_{\epsilon}(f^{n}),1) + C\Delta t \sqrt{\epsilon} \,  \pier{\left(|f^n|,1\right)^2}\leq (\xi_{\epsilon}(f^{n}),1) +C\Delta t \sqrt{\epsilon}.
\end{align*}
Summing over $n\in\{0,\dots,N-1\}$, we obtain that
\begin{equation}
    \label{wp43new}
    (|f^{N}|,1)\leq (\xi_{\pier{\epsilon}}(f^{N}),1)\leq (\xi_{\epsilon}(f^{0}),1)+C\sqrt{\epsilon},
\end{equation}
from which we \pier{infer} the first estimate of \eqref{wp42new}. We continue by taking $\chi=f^{n+1}$ in \eqref{wp26}, obtaining, upon integration by parts and employment of the property $(i)$ of Lemma~\ref{lem1d}, using moreover 
the Cauchy--Schwarz inequality and property $(i)$ of Lemma~\ref{lem2}, that
\begin{align*}
    &\frac{\lVert f^{n+1}\rVert^2}{2}+\frac{\lVert f^{n+1}-f^n\Vert^2}{2}+\Delta t\sqrt{\epsilon}\lVert \nabla f^{n+1}\rVert^2\leq \frac{\lVert f^{n}\rVert^2}{2}-\Delta t \left( \Sigma_{\theta}^{\delta}(f^n)f^{n+1},\nabla_x f^{n+1}\right)\\
    & \quad = \frac{\lVert f^{n}\rVert^2}{2}+\frac{\Delta t}{2} \abramo{\int_0^1\int_{\Omega}}\div \Sigma_{\theta}^{\delta}(f^n)\left(f^{n+1}\right)^2\abramo{\,d\mathbf{x}dc}\leq \frac{\lVert f^{n}\rVert^2}{2}\\
    & \qquad +C\Delta t \lVert \div \Sigma_{\theta}^{\delta}(f^n)\rVert_{L^{\infty}(\Omega\times (0,1))}\lVert f^{n+1}\rVert^2\leq \frac{\lVert f^{n}\rVert^2}{2}+\frac{C}{\min(\theta,\delta)}\Delta t \lVert f^{n+1}\rVert^2,
\end{align*}
Summing over $n\in\{0,\dots,N-1\}$ and applying a discrete Gronwall inequality we obtain that
\begin{equation}
    \label{wp44}
    \displaystyle \sup_{n\in\{1,\dots,N\}}\frac{\lVert f^{n}\rVert^2}{2}\leq \frac{\lVert f^{0}\rVert^2}{2}\mathrm{e}^{\frac{CT}{\min(\theta,\delta)}}\leq C(T,\theta,\delta),
\end{equation}
where $C(T,\theta,\delta)\to +\infty$ as $\theta,\delta \to 0^+$.
Taking now $\chi=-\Delta f^{n+1}=-\Delta_x f^{n+1}-\frac{\partial^2 f^{n+1}}{\partial c^2}$ in \eqref{wp26},  upon integration by parts and employing the boundary conditions \eqref{wp25d}, using moreover \eqref{wp30cc}, 
we obtain that
\begin{align*}
    &\frac{\lVert \nabla f^{n+1}\rVert^2}{2}+\frac{\lVert\nabla( f^{n+1}-f^n)\Vert^2}{2}+\Delta t\sqrt{\epsilon}\lVert \Delta f^{n+1}\rVert^2+\Delta t \left(D(c)\Delta_x f^{n+1},\Delta_x f^{n+1} \right)\\
    & \qquad +\Delta t \left(D(c)\nabla_x \left(\partial_c f^{n+1}\right), \nabla_x \left(\partial_c f^{n+1}\right) \right) = \frac{\lVert \nabla f^{n}\rVert^2}{2} +\frac{\Delta t}{2}\abramo{\int_0^1\int_{\Omega}}D''(c)\left|\nabla_x f^{n+1}\right|^2\abramo{\,d\mathbf{x}dc}\\
    & \qquad -\underbrace{\Delta t\left(\div (\Sigma_{\theta}^{\delta}(f^n))f^{n+1},\Delta f^{n+1} \right)}_{I_1}-\underbrace{\Delta t\left(\Sigma_{\theta}^{\delta}(f^n)\cdot \nabla_xf^{n+1},\Delta f^{n+1} \right)}_{I_2}.
\end{align*}
After an additional integration by parts and employing the boundary conditions \eqref{wp25d} we \pier{infer} that
\[
I_1=-\Delta t \left(\nabla \left[\div (\Sigma_{\theta}^{\delta}(f^n))\right]f^{n+1},\nabla f^{n+1}\right)-\Delta t\left( \div (\Sigma_{\theta}^{\delta}(f^n))\nabla f^{n+1},\nabla f^{n+1}\right),
\]
from which, using \eqref{wp44}, property $(i)$ of Lemma~\ref{lem2}, the Cauchy--Schwarz and the Young inequalities we conclude that
\begin{align*}
&|I_1|\leq \frac{C}{\min(\theta^2,\delta^2)}\Delta t \lVert f^{n+1} \rVert^2+\frac{C}{\min(\theta^2,\delta^2)}\Delta t \lVert \nabla f^{n+1} \rVert^2+\frac{C}{\min(\theta,\delta)}\Delta t \lVert \nabla f^{n+1} \rVert^2\leq \frac{C(T,\theta,\delta)}{\min(\theta^2,\delta^2)}\Delta t\\
& \qquad +\frac{C}{\min(\theta^2,\delta^2)}\Delta t \lVert \nabla f^{n+1} \rVert^2.
\end{align*}
Similarly, upon integration by parts and employment of the boundary conditions \eqref{wp25d} and of the property $(i)$ of Lemma~\ref{lem1d}, we write
\begin{align}
\label{i2}
& \notag I_2=-\Delta t\left(\nabla [\Sigma_{\theta}^{\delta}(f^n)] \nabla_xf^{n+1},\nabla f^{n+1} \right)-\Delta t\abramo{\int_0^1\int_{\Omega}}\Sigma_{\theta}^{\delta}(f^n)\cdot \nabla_x\frac{\left|\nabla f^{n+1}\right|^2}{2}\abramo{\,d\mathbf{x}dc}\\
& \quad =-\Delta t\left(\nabla [\Sigma_{\theta}^{\delta}(f^n)] \nabla_xf^{n+1},\nabla f^{n+1} \right)+\Delta t\abramo{\int_0^1\int_{\Omega}}\div \Sigma_{\theta}^{\delta}(f^n)\frac{\left|\nabla f^{n+1}\right|^2}{2}\abramo{\,d\mathbf{x}dc}.
\end{align}
Hence, employing the property $(i)$ of Lemma~\ref{lem2}, the Cauchy--Schwarz and the Young inequalities we \pier{deduce} that
\[
|I_2|\leq \frac{C}{\min(\theta,\delta)}\Delta t \lVert \nabla f^{n+1} \rVert^2.
\]
Collecting all the previous results we conclude that
\begin{align*}
    &\frac{\lVert \nabla f^{n+1}\rVert^2}{2}+\frac{\lVert\nabla( f^{n+1}-f^n)\Vert^2}{2}+\Delta t\sqrt{\epsilon}\lVert \Delta f^{n+1}\rVert^2+\Delta t \left(D(c)\Delta_x f^{n+1},\Delta_x f^{n+1} \right)\\
    & \qquad +\Delta t \left(D(c)\nabla_x \left(\partial_c f^{n+1}\right), \nabla_x \left(\partial_c f^{n+1}\right) \right) \leq \frac{\lVert \nabla f^{n}\rVert^2}{2} + \frac{C(T,\theta,\delta)}{\min(\theta^2,\delta^2)}\Delta t+\frac{C}{\min(\theta^2,\delta^2)}\Delta t\lVert \nabla f^{n+1}\rVert^2,
\end{align*}
from which, summing over $n\in\{0,\dots,N-1\}$ and applying a discrete Gronwall inequality we obtain the second, third and fourth inequalities in \eqref{wp42new}. The fifth inequality in \eqref{wp42new} can be obtained starting from \eqref{wp30d}, summing over $n\in\{0,\dots,N-1\}$ and employing the second property in \eqref{wp42new}.
\end{proof}
The estimates obtained in Lemma~\ref{lem5} lead to the convergence results for the time interpolants \eqref{wp34} which let us identify the limit of $\mathcal{P}_n^{\theta,\delta,\epsilon}$ as $\Delta t,\epsilon \to 0$. In particular, we consider $\epsilon\equiv \Delta t$ and reduce the limit problem to the case $\Delta t\to 0$ only, obtaining an existence result for problem $\mathcal{P}^{\theta,\delta}$. \abramo{We obtain the following convergence results.}
\begin{lemma}
    \label{lem7}
    Under the Assumptions \pier{$(A2)$--$(A3)$} of Theorem~\ref{thm1} there exists a limit function $f$, with regularity
    \[
    f\in L^{\infty}(0,T;H^1(\Omega \times [0,1]))\cap  H^1\bigl(0,T;L^2(\Omega \times (0,1))\bigr), \; \sqrt{D(c)}\Delta_xf\in L^{2}(0,T;L^2(\Omega \times (0,1)))
    \]
    such that, up to subsequences of the solution of $\mathcal{P}_n^{\theta,\delta,\Delta t}$, which we still label by the index $\Delta t$, the following convergence results hold as $\Delta t \to 0$:
    \begin{align}
    \label{conv7t} & \hat{f}_{\Delta t} \overset{\ast}{\rightharpoonup} f \quad \text{in} \quad \pier{L^{\infty}\bigl(0,T;H^1(\Omega \times [0,1])\bigr)},\\
    \label{conv8t} & \partial_t\hat{f}_{\Delta t} {\rightharpoonup} \partial_tf \quad \text{in} \quad L^{2}\bigl(0,T;L^2(\Omega \times (0,1))\bigr),\\
    \label{conv9t} & \hat{f}_{\Delta t} {\rightarrow} f \quad \text{in} \quad C^{0}\left(0,T;(L^r(\Omega \times (0,1))\right)\;\; \text{and} \;\; \text{a.e. in} \; \; \Omega\times (0,1)\times (0,T),\\
    \label{conv10t} & \bar{f}_{\Delta t},\ubar{f}_{\Delta t} \overset{\ast}{\rightharpoonup} f \quad \text{in} \quad \pier{L^{\infty}\bigl(0,T;H^1(\Omega \times [0,1])\bigr)},\\
    \label{conv10bt} & \sqrt{D(c)}\Delta_x\bar{f}_{\Delta t}{\rightharpoonup} \sqrt{D(c)}\Delta_xf \quad \text{in} \quad  L^{2}\left(0,T;L^2({\Omega} \times (0,1))\right),\\
    \label{conv10ct} & \abramo{\sqrt[4]{\epsilon}}\Delta\bar{f}_{\Delta t}{\rightharpoonup} 0 \quad \text{in} \quad  L^{2}\left(0,T;L^2({\Omega} \times (0,1))\right),\\
    \label{conv11t} & \bar{f}_{\Delta t},\ubar{f}_{\Delta t} {\rightarrow} f \quad \text{in} \quad L^{\infty}\left(0,T;L^r(\Omega \times (0,1))\right)\;\; \text{and} \;\; \text{a.e. in} \; \; \Omega\times (0,1)\times (0,T),\\
    \label{conv12t} & \Sigma_{\theta}^{\delta}\left(\ubar{f}_{\Delta t}\right) {\rightarrow} \Sigma_{\theta}^{\delta}\left(f\right) \quad \text{in} \quad L^{\infty}\left(\Omega \times (0,1)\times (0,T);\mathbb{R}^d\right),\\
    \label{conv13t} & \div \Sigma_{\theta}^{\delta}\left(\ubar{f}_{\Delta t}\right) {\rightarrow} \div \Sigma_{\theta}^{\delta}\left(f\right) \quad \text{in} \quad L^{\infty}\left(0,T;L^r(\Omega \times (0,1))\right),
    \end{align}
    where $r\geq 1$ for $d=2$ and $r\in [1,6)$ for $d=3$.
\end{lemma}
\abramo{The convergence results \eqref{conv7t}-\eqref{conv10t} and \eqref{conv11t}-\eqref{conv13t} in Lemma \ref{lem7} follow directly from the estimates of Lemma~\ref{lem5} in the same way as \eqref{conv7}--\eqref{conv13} follow from the estimates in Lemma~\ref{lem3}, \abramo{with the difference that they are now $\epsilon$-independent}, so for the sake of brevity we state them without \pier{giving an explicit proof}. The only new convergence results with respect to the analogous results reported in \eqref{conv7}--\eqref{conv13} are \eqref{conv10bt} and \eqref{conv10ct}, which are direct consequences of the third and fourth estimates in \eqref{wp42new}.}

We are now ready to prove that the limit function $f$ obtained in Lemma~\ref{lem7} is the unique strong solution, under Assumptions \pier{$(A2)$--$(A3)$}, of \eqref{wp19}-\eqref{wp21} (with $\epsilon=0$), as defined in Theorem~\ref{thm6}. 
\begin{proof}[\pier{Proof of Theorem~\ref{thm6}}]
Starting from \eqref{wp39} and using the convergence properties of Lemma~\ref{lem7}, \pier{it is} easy to obtain that
\begin{align*}
    & \int_0^T\left(\partial_t \hat{f}_{\Delta t},g\right)\abramo{\,dt} \rightarrow \int_0^T\left(\partial_t f,g\right)\abramo{\,dt}, \quad \int_0^T(D(c)\Delta_x\bar{f}_{\Delta t}, g)\abramo{\,dt}\rightarrow \int_0^T(D(c)\Delta_xf, g)\abramo{\,dt},\\
    & \int_0^T\left(\div \left(\Sigma_{\theta}^{\delta}(\ubar{f}_{\Delta t})\bar{f}_{\Delta t}\right), g\right)\abramo{\,dt}\rightarrow \int_0^T\left(\div \left(\Sigma_{\theta}^{\delta}(f)f\right), g\right)\abramo{\,dt},\quad \sqrt{\epsilon}\int_0^T\left(\Delta \bar{f}_{\Delta t},g\right)\abramo{\,dt}\rightarrow 0,
\end{align*}
as $\Delta t \to 0$, where for calculating the first limit in the second row we use the same arguments as in the proof of Theorem~\ref{thm1}. Hence, we have proved that the limit point $f$ satisfies \eqref{wp19}-\eqref{wp21}, with $\epsilon=0$, a.e. in $\Omega\times [0,1]\times [0,T]$, i.e.
\begin{equation}
    \label{wp19t}
    \partial_tf-\div \left(\Sigma_{\theta}^{\delta}(f)f\right)-D(c)\Delta_xf=0
\end{equation}
is valid a.e. in $\Omega\times [0,1]\times [0,T]$, with $f$ enjoying the regularity \eqref{wp41new} and with \eqref{wp20} (with $\epsilon=0$) valid in the sense of traces and a.e. and $f(\vec{x},c,0)=f_0(\vec{x},c)$ a.e. in $\Omega \times [0,1]$. Thanks to Fatou's lemma applied to the first estimate in \eqref{wp42new}, passing to the limit as $\epsilon\equiv \Delta t\to 0$ we have that
\begin{equation}
\label{wp19t2}
\int_0^1\int_{\Omega}|f|\abramo{\,d\mathbf{x}dc}\leq \int_0^1\int_{\Omega}f_0\abramo{\,d\mathbf{x}dc}=1.
\end{equation}
Moreover, taking the integral of \eqref{wp19t} over $\Omega\times [0,1]$ and using the boundary conditions, we have that $\int_0^1\int_{\Omega}f=1$. Hence, we conclude as before that $f\geq 0$ a.e. in $\Omega\times [0,1]$.

For what concerns the continuous dependence from data \eqref{wp43s2}, let us assume that $f_1^0,f_2^0$ $\in H^{1}(\Omega\times [0,1])$ be two different initial data, and let $f_1,f_2$ be the two corresponding strong solutions. We define $\tilde{f}:=f_1-f_2$. 
Taking the difference between \eqref{wp19t} for $f_1$ and $f_2$, multiplying the resulting equation by $ \tilde{f}$ and integrating over $\Omega \times [0,s]$, for $s\in (0,T]$, we obtain that
\begin{align*}
   & \frac{\lVert \tilde{f}\rVert^{\pier{2}}(s)}{2}+\int_0^s\abramo{\int_0^1\int_{\Omega}}D(c)|\nabla_x \tilde{f}|^2\abramo{\,d\mathbf{x}dc}
   = \frac{\lVert \tilde{f}\rVert^{\pier{2}}(0)}{2}+\int_0^s\abramo{\int_0^1\int_{\Omega}}\div \Sigma_{\theta}^{\delta}(f_1) \,|\tilde{f}|^2\abramo{\,d\mathbf{x}dcdt}\\
   & \qquad +\frac{1}{2}\int_0^s\abramo{\int_0^1\int_{\Omega}}\Sigma_{\theta}^{\delta}(f_1) \cdot\nabla_x \left(\tilde{f}^2\right)\abramo{\,d\mathbf{x}dcdt}+\int_0^s\left(\div \Sigma_{\theta}^{\delta}(\tilde{f}) f_2,\tilde{f}\right)\abramo{\,dt}+\int_0^s\left(\Sigma_{\theta}^{\delta}(\tilde{f})\cdot \nabla_x f_2,\tilde{f}\right)\abramo{\,dt}.
\end{align*}
Using the Cauchy--Schwarz and Young inequalities, the property $(i)$ 
of Lemma~\ref{lem2}, \eqref{wp41new} and integration by parts, we easily \pier{infer} that
\begin{align*}
   & \frac{\lVert \tilde{f}\rVert^{\pier{2}}(s)}{2}+\int_0^s\left(D(c),|\nabla_x \tilde{f}|^2\right)\abramo{\,dt}
   \leq \frac{\lVert \tilde{f}\rVert^{\pier{2}}(0)}{2}+\frac{C}{\min(\theta,\delta)}\int_0^s\lVert \tilde{f}\rVert^2\abramo{\,dt}.
\end{align*}
An application of the Gronwall inequality to the previous estimate leads to \eqref{wp43s2}. We observe that \pier{it is} not possible to obtain a strong continuous dependence from data like in \pier{\eqref{wp23s}
,} since we cannot take the $L^2$ product of the difference between \eqref{wp19t} for $f_1$ and $f_2$ and $-\Delta \tilde{f}$, due to the fact that $\Delta \tilde{f}\notin L^2(\Omega \times (0,1))$. 
\end{proof}

As for the limit case $\theta,\delta\to 0$, under Assumption $(A1bis)$ \pier{it is} possible to enhance the regularity of the solution by getting uniform in $\epsilon$ a-priori $L^p$ estimates to obtain the following regularity result. 
\begin{theorem}
    \label{thm7}
    \betti{Let $\theta,\delta>0$ and Assumption $(A1bis)$ hold true}. Then, the strong solution of Theorem~\ref{thm6} enjoys the further regularity
    \begin{equation}
        \label{wp45}
    f\in L^{\infty}(0,T;L^{p}(\Omega\times (0,1))),
    \end{equation}
    for any $p\in [1,+\infty)$.
\end{theorem}
\begin{proof}
Let us multiply \eqref{wp19t} by $pf^{p-1}$ and integrate over $\Omega \times [0,1]$. With simple calculations, taking integration by parts and employing the boundary conditions \eqref{wp20} (with $\epsilon=0$) and the 
property $(i)$ of Lemma~\ref{lem1d}, using moreover the property $(i)$ of 
Lemma~\ref{lem2}, we obtain that
\begin{align*}
    & p\left(\partial_t f,f^{p-1}\right)-p\left(D(c)\Delta_xf,f^{p-1}\right)=\frac{d}{dt}\left(f^p,1\right)+4\frac{(p-1)}{p}\left(D(c)\nabla_x \left(f^{p/2}\right),\nabla_x \left(f^{p/2}\right)\right)\\
    & \quad =p\left(\div\left(\Sigma_{\theta}^{\delta}(f)\right),f^p\right)+\left(\Sigma_{\theta}^{\delta}(f),\nabla_x\left(f^p\right)\right)= p\left(\div\left(\Sigma_{\theta}^{\delta}(f)\right),f^p\right)-\left(\div \Sigma_{\theta}^{\delta}(f),f^p\right)\\
    & \quad \leq \frac{C}{\min(\theta,\delta)}p\left(f^p,1\right).
\end{align*}
Integrating in time over the interval $[0,T]$ and using a Gronwall argument we obtain \eqref{wp45}.
\end{proof}
\begin{remark}
    \label{mosal}
    We observe that the Moser--Alikakos iteration argument is not working for problem $\mathcal{P}^{\theta,\delta,0}$, due to the lack of regularity of $\nabla f$ and $\nabla (f^{p/2})$, hence we are not able to prove an $L^{\infty}(0,T;L^{\infty}(\Omega\times (0,1)))$ bound for the solution of $\mathcal{P}^{\theta,\delta,0}$, which was instead possible for problem $\mathcal{P}^{0,0,\epsilon}$. 
\end{remark}
\subsection{Limit problem $\mathcal{P}^{\theta,\delta,\epsilon}\to \mathcal{P}$ as $\epsilon,\theta,\delta\to 0$}
\pier{Finally, we study} the limit problem as $\epsilon,\theta,\delta \to 0$, giving an existence result for the case without any regularization in the equation. We observe that all estimates but the first in \eqref{wp42new} depend on inverse powers of the regularization parameters. The first estimate in \eqref{wp42new} implies that the strong solution of Theorem~\ref{thm6} is uniformly in $\epsilon$, $\theta$ and $\delta$ bounded in the space $L^{\infty}(0,T;L^1(\Omega\times (0,1)))$, with $f\geq 0$ a.e. in $\Omega\times (0,1)$. \pier{It is} well known that, since $L^1(\Omega\times (0,1))$ is neither reflexive nor the dual of a Banach space, uniform boundedness of sequences in $L^1(\Omega\times (0,1))$ does not imply their weak relative compactness. Hence, we need to obtain further a-priori estimates which are uniform in $\epsilon$, $\theta$ and $\delta$ to study the limit as $\epsilon,\theta,\delta \to 0$.


\pier{Setting $\theta \equiv \delta$, we first perform the limit passage $\mathcal{P}^{\theta,\theta,\epsilon}\to \mathcal{P}^{\theta,\theta}$ as $\epsilon\to 0$, which has already been analyzed in the previous section. By Theorem~\ref{thm6}, there exists a unique strong solution, denoted by $f_{\theta}$ to emphasize its dependence on the parameter $\theta$, with the regularity stated in \eqref{wp41new}. Moreover, $f_{\theta}$ satisfies \eqref{wp19}--\eqref{wp21} (with $\epsilon=0$) almost everywhere; that is,}
\begin{equation}
    \label{wp190}
    \partial_tf_{\theta}-\div \left(\Sigma_{\theta}^{\delta}(f_{\theta})f_{\theta}\right)-D(c)\Delta_xf_{\theta}=0 \quad \text{a.e. in} \;\; \Omega\times [0,1]\times[0,T],
\end{equation}
with homogeneous Neumann boundary conditions
\begin{equation}
    \label{wp200}
    \Sigma_{\theta}^{\delta}(f_{\theta})f_{\theta}\cdot \vec{n}+D(c)\nabla_xf_{\theta}\cdot \vec{n}=0\quad \text{a.e. on} \; \partial \Omega\times [0,1]\times[0,T], 
\end{equation}
and with the initial condition
\begin{equation}
    \label{wp210}
    f_{\theta}(\cdot,\cdot,0)=f_0(\cdot,\cdot)\quad \text{a.e. in}\;\;\Omega\times(0,1).
\end{equation}
\pier{As a second step, we pass to the limit as $\theta \to 0$ in the weak formulation of \eqref{wp190}, exploiting the convergence results established in Lemma~\ref{lem6} below. These results are derived under the following assumption on the initial condition:}
\begin{itemize}
    \item[(A3bis)] $f_0\in H^1(\Omega\times [0,1])\cap L^{\infty}(\Omega\times (0,1))$, with $f_0\geq 0$ a.e. in $\Omega \times [0,1]$ and $\int_0^1\int_{\Omega}f_0d\vec{x}dc=1$.
\end{itemize}
\begin{lemma}
    \label{lem6}
    Under the Assumptions $(A2)-(A3bis)$ there exists a limit function $f$, with regularity
    \[
    f\in L^{\infty}\bigl(0,T;L^1(\Omega \times (0,1))\bigr),
    \]
    such that, up to subsequences of the solution \pier{to} $\mathcal{P}^{\theta,\theta}$, which we still label by the index $\theta$, the following strong convergence results hold as $\theta \to 0$: 
    \begin{align}
    \label{conv14} & {f}_{\theta} {\rightarrow} f \quad \text{in} \quad \abramo{L^2}\left(0,T;(L^r(\Omega \times (0,1)_{\text{loc}})\right)\;\; \text{and} \;\; \text{a.e. in} \; \; \Omega\times (0,1)\times (0,T),\\
    \label{conv15} & {f}_{\theta} {\rightarrow} f \quad \text{in} \quad \abramo{L^{2}}\left(0,T;\pier{L^1(\Omega \times (0,1))}\right),
    \end{align}
    where $r\geq 1$ for $d=2$ and $r\in [1,6)$ for $d=3$.
\end{lemma}
\begin{proof}
We start by recalling that \eqref{wp19t2}, together with the property that $f_{\theta}\geq 0$ a.e. in $\Omega \times [0,1]\times (0,T)$, is \pier{satisfied for every} $\theta$.

For a given $\iota>0$, let us multiply \eqref{wp190} by $f_{\theta}$ and integrate over $\Omega \times (2\iota,1-2\iota)\times(0,s)$, for any given $s\in (0,T]$. Integrating by parts in space and employing the boundary condition \eqref{wp200}, we obtain that
\begin{align*}
& \frac{1}{2}\int_{2\iota}^{1-2\iota}\!\!\!\int_{\Omega}f_{\theta}^2(s)\abramo{\,d\mathbf{x}dc}+\int_0^s\int_{2\iota}^{1-2\iota}\!\!\!\int_{\Omega}D(c)|\nabla_x f_{\theta}|^2\abramo{\,d\mathbf{x}dcdt}=\frac{1}{2}\int_{2\iota}^{1-2\iota}\!\!\!\int_{\Omega}f_{\theta}^2(0)\abramo{\,d\mathbf{x}dc}\\
& \qquad +\int_0^s\int_{2\iota}^{1-2\iota}\!\!\!\int_{\Omega}\Sigma_{\theta}^{\theta}(f_{\theta})f_{\theta}\nabla_xf_{\theta}\abramo{\,d\mathbf{x}dcdt}.
\end{align*}
Using \eqref{wp19t2}, property $(ii)$ of Lemma~\ref{lem1d}, the Cauchy--Schwarz and the Young inequalities and introducing the quantity $\hat{\iota}:=\min_{c\in [2\iota,1-2\iota]}D(c)>0$ we \pier{infer} that
\begin{align*}
    & \frac{1}{2}\int_{2\iota}^{1-2\iota}\!\!\!\int_{\Omega}f_{\theta}^2(s)\abramo{\,d\mathbf{x}dc}+\hat{\iota}\int_0^s\int_{2\iota}^{1-2\iota}\!\!\!\int_{\Omega}|\nabla_x f_{\theta}|^2\abramo{\,d\mathbf{x}dcdt}\leq \frac{1}{2}\int_{2\iota}^{1-2\iota}\!\!\!\int_{\Omega}f_{\theta}^2(0)\abramo{\,d\mathbf{x}dc}\\
    & \qquad +\frac{\hat{\iota}}{2}\int_0^s\int_{2\iota}^{1-2\iota}\!\!\!\int_{\Omega}|\nabla_x f_{\theta}|^2\abramo{\,d\mathbf{x}dcdt}+\frac{C}{\hat{\iota}}\int_0^s\int_{2\iota}^{1-2\iota}\!\!\!\int_{\Omega}f_{\theta}^2\abramo{\,d\mathbf{x}dcdt}.
\end{align*}
An application of the Gronwall inequality then gives that
\begin{equation}
    \label{wp47bis}
    \esssup_{s\in [0,T]}\int_{2\iota}^{1-2\iota}\!\!\!\int_{\Omega}f_{\theta}^2(s)\abramo{\,d\mathbf{x}dc}\leq C(\iota),\;\; \int_0^s\int_{2\iota}^{1-2\iota}\!\!\!\int_{\Omega}|\nabla_x f_{\theta}|^2\abramo{\,d\mathbf{x}dcdt}\leq C(\iota),
\end{equation}
uniformly in $\theta$, with $C(\iota)\to +\infty$ as $\iota \to 0^+$.

We introduce now, for a given $\iota>0$, the cut-off function $\rho_{\iota}\in W^{1,\infty}([0,1])$ such that
    \[
    \rho_{\iota}(\cdot)\equiv 1\; \text{on}\;I_{\iota}:=[2\iota,1-2\iota]; \;\; \pier{\operatorname{supp}}(\rho_{\iota})\subseteq I_{\iota/2} \; \text{ and }\; 0\leq \rho_{\iota}(\cdot)\leq 1;
    \]
and such that
    \[
    \left|\frac{\partial \rho_{\iota}(c)}{\partial c}\right|\leq \frac{C}{\iota}.
    \]
\pier{By combining piecewise linear and constant components, you can easily} construct such cut-off functions~$\rho_{\iota}$. 
We now multiply \eqref{wp190} by $g=-\partial_c(\rho_{2\iota}^2(c)\partial_c f_{\theta})$ and integrate over $\Omega\times (0,1)\times (0,s)$, for any given $s\in (0,T]$. 
\pier{Note that $g\notin L^2(0,s;L^2(\Omega\times(0,1)))$, since Theorem~\ref{thm6} does not provide any control of $\partial_c^2 f_{\theta}$. Therefore, the following computations are only formal. A rigorous justification could be obtained either by introducing suitable mollifications or by deriving the same estimate from the strong formulation of problem $\mathcal{P}^{\theta,\theta,\epsilon}$, with $\epsilon,\theta>0$, exploiting the regularity result \eqref{wp23a} and subsequently passing to the limit as $\epsilon\to0$. For the sake of readability, we omit these technical details and proceed at a formal level.}
We obtain, after integration by parts and employing the boundary conditions \eqref{wp200}, the property $(i)$ of Lemma~\ref{lem1d} and the fact that $\rho_{2\iota}(0)=\rho_{2\iota}(1)=0$, that
\begin{align}
\label{wp47tris}
    & \notag \frac{1}{2}\int_0^1\int_{\Omega}\rho_{2\iota}^2\betti{(c)}|\partial_c f_{\theta}|^2(s)\abramo{\,d\mathbf{x}dc}+\int_0^s\int_0^1\int_{\Omega}D(c)\rho_{2\iota}^2(c)|\nabla_x(\partial_c f_{\theta})|^2 \abramo{\,d\mathbf{x}dcdt}=\frac{1}{2}\int_0^1\int_{\Omega}\rho_{2\iota}^2\betti{(c)}|\partial_c f_{\theta}|^2(0)\abramo{\,d\mathbf{x}dc}\\
    & \qquad -\underbrace{\int_0^s\left(D'(c)\nabla_x f_{\theta},\rho_{2\iota}^2\pier{(c)}\nabla_x(\partial_cf_{\theta})\right)\abramo{\,dt}}_{=:I_1}-\underbrace{\int_0^s\left(\partial_c\left(\Sigma_{\theta}^{\theta}(f_{\theta})f_{\theta}\right),\rho_{2\iota}^2(c)\nabla_x(\partial_cf_{\theta})\right)\abramo{\,dt}}_{=:I_2}.
\end{align}
We introduce the quantity $\bar{\iota}:=\min_{c\in I_{\iota}}D(c)>0$.
Using the Cauchy--Schwarz and the Young inequalities, Assumption~$(A2)$, \pier{the facts that $\pier{\operatorname{supp}}(\rho_{2\iota})\subseteq [2\iota,1-2\iota]$ and $0\leq \rho_{2\iota} \leq 1$, together with} \eqref{wp47bis}, we obtain
\begin{align*}
&|I_1|\leq \frac{\bar{\iota}}{4}\int_0^s\int_0^1\int_{\Omega}\rho_{2\iota}^2(c)|\nabla_x (\partial_c f_{\theta})|^2\abramo{\,d\mathbf{x}dcdt}+\frac{C}{\bar{\iota}}\int_0^s\int_{2\iota}^{1-2\iota}\!\!\!\int_{\Omega}|\nabla_x f_{\theta}|^2\abramo{\,d\mathbf{x}dct}\\
& \quad \leq \frac{\bar{\iota}}{4}\int_0^s\int_0^1\int_{\Omega}\rho_{2\iota}^2(c)|\nabla_x (\partial_c f_{\theta})|^2\abramo{\,d\mathbf{x}dcdt}+C(\iota),
\end{align*}
with $C(\iota)\to +\infty$ as $\iota\to 0^+$.
\abramo{%
For \pier{the treatment of $I_2$ we additionally invoke \eqref{wp19t2} and property $(ii)$ of Lemma~\ref{lem1d}, deducing} that 
\begin{align*}
    & |I_2|\leq \frac{\bar{\iota}}{4}\int_0^s\int_0^1\int_{\Omega}\rho_{2\iota}^2(c)|\nabla_x(\partial_c f_{\theta})|^2\abramo{\,d\mathbf{x}dc}+\frac{C}{\bar{\iota}}\int_0^s\lVert\partial_c\left(\Sigma_{\theta}^{\theta}(f_{\theta})\right)\rVert_{L^{\infty}(\Omega\times (0,1);\mathbb{R}^d)}^2\,dt\esssup_{\tau\in [0,s]}\int_{2\iota}^{1-2\iota}\lVert f_{\theta}(\tau)\rVert_{L^2(\Omega)}^2\,dc\\
    & \qquad +\frac{C}{\bar{\iota}}\int_0^s\lVert \Sigma_{\theta}^{\theta}(f_{\theta})\rVert_{L^{\infty}(\Omega\times (0,1);\mathbb{R}^d)}^2\int_0^1\int_{\Omega}\rho_{2\iota}^2(c)\lvert \partial_c f_{\theta}\rvert^2\,d\mathbf{x}dcdt\leq \frac{\bar{\iota}}{4}\int_0^s\int_0^1\int_{\Omega}\rho_{2\iota}^2(c)|\nabla_x(\partial_c f_{\theta})|^2\abramo{\,d\mathbf{x}dc}dt\\
    & \qquad +C(\iota)\int_0^s\lVert\partial_c\left(\Sigma_{\theta}^{\theta}(f_{\theta})\right)\rVert_{L^{\infty}(\Omega\times (0,1);\mathbb{R}^d)}^2\,dt+C(\iota)\int_0^s\int_0^1\int_{\Omega}\rho_{2\iota}^2(c)\lvert \partial_c f_{\theta}\rvert^2\abramo{\,d\mathbf{x}dcdt},
\end{align*}
where $C(\iota)\to +\infty$ as $\iota\to 0^+$. We now observe from property $(iii)$ of Lemma~\ref{lem1d} and from the fact that $f_{\theta}\geq 0$ a.e. in $\Omega \times [0,1]\times (0,T)$, that 
\[
\lVert\partial_c\left(\Sigma_{\theta}^{\theta}\pier{(f(t))}\right)\rVert_{L^{\infty}(\Omega\times (0,1);\mathbb{R}^d)}\leq C \left\lVert \int_{\Omega}f_{\theta}(x^*,\cdot,t)\,dx^* \right\rVert_{L^{\infty}(0,1)}.
\]
Taking the integral over $\Omega$ of \eqref{wp190} and employing the boundary conditions \eqref{wp200}, we \pier{infer} that
\[
\int_{\Omega}f_{\theta}(x^*,c,t)\,dx^*=\int_{\Omega}f_{\theta}(x^*,c,0)\,dx^* \;\ \text{a.e.} \;\, (c,t)\in (0,1)\times (0,T).
\]
Then, considering Assumption $(A3bis)$, we conclude that}
\[
\abramo{\lVert\partial_c\left(\Sigma_{\theta}^{\theta}\pier{(f(t))}\right)\rVert_{L^{\infty}(\Omega\times (0,1);\mathbb{R}^d)}\leq C \left\lVert \int_{\Omega}f_{\theta}(x^*,\cdot,0)\,dx^* \right\rVert_{L^{\infty}(0,1)}\leq C\lVert f_0\rVert_{L^{\infty}(\Omega \times (0,1))}\leq C.}
\]
Collecting all the results, we have that
\begin{align*}
    & \frac{1}{2}\int_0^1\int_{\Omega}\rho_{2\iota}^2|\partial_c f_{\theta}|^2(s)\abramo{\,d\mathbf{x}dc}+\frac{\bar{\iota}}{2}\int_0^s\int_0^1\int_{\Omega}\rho_{2\iota}^2(c)|\nabla_x(\partial_c f_{\theta})|^2\abramo{\,d\mathbf{x}dcdt}\leq \frac{1}{2}\int_0^1\int_{\Omega}\rho_{2\iota}^2|\partial_c f_{\theta}|^2(0)\abramo{\,d\mathbf{x}dc}\\
    & \qquad +C(\iota)+C(\iota)\int_0^s\int_0^1\int_{\Omega}\rho_{2\iota}^2|\partial_c f_{\theta}|^2\abramo{\,d\mathbf{x}dc}dt,
\end{align*}
from which, using a Gronwall argument and \eqref{wp47bis}, we conclude that 
\begin{equation}
\label{wp47b}
\abramo{\int_0^T}\pier{\int_{4\iota}^{1-4\iota}\!\!\!\int_{\Omega}}|\nabla f_{\theta}|^2\abramo{\,d\mathbf{x}dcdt}\leq C(\iota),
\end{equation}
uniformly in $\theta$, \betti{with} $C(\iota)\to +\infty$ as $\iota\to 0^+$. From a comparison argument in \eqref{wp190} and thanks to \eqref{wp47bis}-\eqref{wp47b}, we also have that, for any $\iota>0$,
\begin{equation}
\label{wp47c}
\lVert \partial_t f_{\theta}\rVert_{\abramo{L^{2}}\left(0,T;\left(H^1(\Omega\times [4\iota,1-4\iota])\right)'\right)}\leq C,
\end{equation}
uniformly in $\theta$.
\abramo{Indeed, from \eqref{wp190} and thanks to the regularity \eqref{wp41new} of $f_{\theta}$, we can write
\begin{align*}
&<\partial_t f_{\theta},\chi>_{\Omega \times [4\iota,1-4\iota]}{}=\pier{\int_{4\iota}^{1-4\iota}\!\!\!\int_{\Omega}}\partial_t f_{\theta}\,\chi\abramo{\,d\mathbf{x}dc}=-{}\pier{\int_{4\iota}^{1-4\iota}\!\!\!\int_{\Omega}}\Sigma_{\theta}^{\theta}(f_{\theta})f_{\theta}\cdot\nabla_x \chi\abramo{\,d\mathbf{x}dc}\\
& \qquad -\pier{\int_{4\iota}^{1-4\iota}\!\!\!\int_{\Omega}}D(c)\nabla_xf_{\theta}\cdot \nabla_x \chi\abramo{\,d\mathbf{x}dc}
\end{align*}
for a.e. $t\in (0,T)$ and for any $\chi \in L^2(0,T;H^1(\Omega \times [4\iota,1-4\iota]))$,
\pier{Then, in view of  \eqref{wp47bis}-\eqref{wp47b} we easily deduce} that
\[
\int_0^T\lvert <\partial_t f_{\theta},\chi>_{\Omega \times [4\iota,1-4\iota]}\rvert \,dt\leq C\lVert \chi\rVert_{L^2(0,T;H^1(\Omega \times [4\iota,1-4\iota]))},
\]
which gives \eqref{wp47c}.} 
Hence, the Aubin--Lions theorem gives the compactness result \eqref{conv14}. Also, since for any $c\in (0,1)$ there exists a $\iota>0$ such that $c\in (4\iota,1-4\iota)$, the compactness result in \eqref{conv14} implies the pointwise convergence a.e. in $\Omega \times (0,1)\times (0,T)$, \pier{at least for another subsequence.} \abramo{Moreover, \pier{owing to Fatou's lemma,} the pointwise convergence $f_{\theta}\to f$ a.e. in $\Omega \times (0,1)\times (0,T)$ and the uniform bound $\lVert f_{\theta}\rVert_{L^{\infty}(0,T;L^1(\Omega\times (0,1)))}\leq C$ imply that also the limit $f\in L^{\infty}(0,T;L^1(\Omega\times (0,1)))$.}

In order to prove \eqref{conv15}, we preliminarily multiply \eqref{wp190} by $g\equiv \frac{1}{\sqrt[4]{c(1-c)}}$ and integrate over $\Omega\times (0,1)\times (0,s)$, for any given $s\in (0,T]$. Note that $g$ is admissible since it belongs to $L^2(0,1)$. We obtain that
\begin{align*}
    & \int_0^1\int_{\Omega}\frac{1}{\sqrt[4]{c(1-c)}}f_{\theta}(s)\abramo{\,d\mathbf{x}dc}=\int_0^1\int_{\Omega}\frac{1}{\sqrt[4]{c(1-c)}}f_{\theta}(0)\abramo{\,d\mathbf{x}dc}.
\end{align*}
Hence, thanks to Assumption $(A3bis)$, we obtain that
\begin{align*}
    & \int_0^1\int_{\Omega}\frac{1}{\sqrt[4]{c(1-c)}}f_{\theta}(s)\abramo{\,d\mathbf{x}dc}\leq |\Omega|\lVert f_{\theta}(0) \rVert_{L^{\infty}(\Omega\times (0,1))} \int_0^1\frac{1}{\sqrt[4]{c(1-c)}}\abramo{\,dc}\leq C,
\end{align*}
which implies that
\begin{equation}
\label{wp48}
\esssup_{s\in [0,T]}\int_0^1\int_{\Omega}\frac{1}{\sqrt[4]{c(1-c)}}f_{\theta}(s)\abramo{\,d\mathbf{x}dc}\leq C.
\end{equation}
    We point out that estimate \eqref{wp48} is valid for the solution of system $\mathcal{P}^{\theta,\theta}$, while it would not be possible to prove a similar estimate for the solution of system $\mathcal{P}^{\theta,\theta,\epsilon}$. This is the reason why we need to perform the limit passage in two consecutive steps, setting firstly $\epsilon\to 0$ and secondly $\theta\to 0$. 

Thanks to the Fatou's lemma we can extend the property~\eqref{wp48} also to the limit point $f$ \pier{that is} identified by \eqref{conv14}, i.e.,
\begin{equation}
\label{wp48b}
\esssup_{s\in [0,T]}\int_0^1\int_{\Omega}\frac{1}{\sqrt[4]{c(1-c)}}f(s)\abramo{\,d\mathbf{x}dc}\leq C.
\end{equation}
\pier{We are now in a position to prove \eqref{conv15}. Let $0<\iota<1$. We find out that}
\abramo{%
\begin{align*}
&\lVert f_{\theta}-f \rVert_{\abramo{L^{2}}(0,T;L^1(\Omega \times (0,1)))}^2\\
& \quad\leq 2\lVert f_{\theta}-f \rVert_{\abramo{L^{2}}(0,T;L^1(\Omega \times (\iota,1-\iota)))}^2+2\int_0^T\left(\int_0^{\iota}\int_{\Omega}|f_{\theta}-f|\abramo{\,d\mathbf{x}dc}\right)^2\abramo{\,dt}+2\int_0^T\left(\int_{1-\iota}^{1}\int_{\Omega}|f_{\theta}-f|\abramo{\,d\mathbf{x}dc}\right)^2\abramo{\,dt}\\
& \quad = 2\lVert f_{\theta}-f \rVert_{L^{2}(0,T;L^1(\Omega \times (\iota,1-\iota)))}^2+2\int_0^T\left(\int_0^{\iota}\int_{\Omega}\frac{\sqrt[4]{c(1-c)}}{\sqrt[4]{c(1-c)}}|f_{\theta}-f|\abramo{\,d\mathbf{x}dc}\right)^2\abramo{\,dt}\\
& \qquad+2\int_0^T\left(\int_{1-\iota}^{1}\int_{\Omega}\frac{\sqrt[4]{c(1-c)}}{\sqrt[4]{c(1-c)}}|f_{\theta}-f|\abramo{\,d\mathbf{x}dc}\right)^2\abramo{\,dt}\leq 2\lVert f_{\theta}-f \rVert_{L^{2}(0,T;L^1(\Omega \times (\iota,1-\iota)))}^2\\
& \qquad +2\sqrt{\frac{\iota}{2}}T\left[\left(\esssup_{s\in [0,T]}\int_0^{\iota}\int_{\Omega}\frac{1}{\sqrt[4]{c(1-c)}}|f_{\theta}-f|\abramo{\,d\mathbf{x}dc}\right)^2+\left(\esssup_{s\in [0,T]}\int_{1-\iota}^{1}\int_{\Omega}\frac{1}{\sqrt[4]{c(1-c)}}|f_{\theta}-f|\abramo{\,d\mathbf{x}dc}\right)^2\right]\\
& \quad \rightarrow  0
\end{align*}}%
as $\iota \to 0$, thanks to \eqref{conv14}, \eqref{wp48} and \eqref{wp48b}. This proves \eqref{conv15}.
\end{proof}
With the strong convergence results introduced in Lemma~\ref{lem6} we can characterize the limit problem $\mathcal{P}^{\theta,\theta}\to \mathcal{P}$ as $\theta\to 0$, giving the following existence result for problem $\mathcal{P}$. 
\begin{theorem}
    \label{thm8}
    \betti{Let} Assumptions $(A2)-(A3bis)$ hold \betti{true}. Then, there exists a distributional solution $f$ of problem $\mathcal{P}$, with 
    \begin{equation}
        \label{wp49}
        f\in L^{\infty}(0,T;L^1(\Omega \times (0,1))),
    \end{equation}
    which satisfies the formulation 
    \begin{align}
    \label{wp50}
   -\int_0^T\abramo{\int_0^1\int_{\Omega}}f\,\zeta ' h\abramo{\,d\mathbf{x}dc}+\int_0^T\abramo{\int_0^1\int_{\Omega}}\Sigma(f)f\cdot \zeta \nabla_x h\abramo{\,d\mathbf{x}dc}-\int_0^T\abramo{\int_0^1\int_{\Omega}}D(c) f\, \zeta\Delta_xh\abramo{\,d\mathbf{x}dc}-\left( f_0,\zeta(0)h\right)=0,
\end{align}
for all $\zeta \in C_c^1([0,T))$, $h\in L^{\infty}(\Omega \times (0,1))$ with $\nabla_x h \in L^{\infty}(\Omega \times (0,1);\mathbb{R}^d)$, $\Delta_x h \in L^{\infty}(\Omega \times (0,1))$
 and such that $\partial_nh(\cdot,c)|_{\partial \Omega}=0$, with the further properties that $f\geq 0$ a.e. in $\Omega \times [0,1]\times [0,T]$ and $\int_0^1\int_{\Omega}fd\vec{x}dc=1$.
\end{theorem}
\abramo{
\begin{remark}
    We remark that the initial condition \pier{for $f$, that is \eqref{wp3}, is} implicitly included in the weak-distributional formulation \eqref{wp50}.
\end{remark}
}
\begin{proof}
\pier{Let $g=\zeta h$, where $\zeta$ and $h$ are as in the statement. Multiplying \eqref{wp190} by $g$ and integrating over $\Omega\times(0,1)\times(0,T)$, an integration by parts in both time and space yields}
    \begin{align}
    \label{wp46lim}
    -\int_0^T\left(f_{\theta},\zeta'h\right)+\int_0^T(\Sigma_{\theta}^{\theta}(f_{\theta})f_{\theta}, \zeta \nabla_x h)-\int_0^T(D(c)f_{\theta}, \zeta \Delta_x h)-(f_0,\zeta(0)h)=0.
\end{align}
We employ the convergence result \eqref{conv15} to pass to the limit in \eqref{wp46lim} and obtain \eqref{wp50}. The only non-trivial steps in the limit passage regard the second term in \eqref{wp46lim}, so we only show the details for it. 
We rewrite the second term in \eqref{wp46lim} as
\begin{align*}
 & \int_0^T(\Sigma_{\theta}^{\theta}(f_{\theta})f_{\theta}, \zeta \nabla_x h)\abramo{\,dt}=\underbrace{\int_0^T\left(\left(\Sigma_{\theta}^{\theta}(f_{\theta})-\Sigma_{\theta}^{\theta}(f)\right)f_{\theta}, \zeta \nabla_x h\right)\abramo{\,dt}}_{=:I_1}+\underbrace{\int_0^T\abramo{\int_0^1\int_{\Omega}}\left(\Sigma_{\theta}^{\theta}(f)-\Sigma(f)\right)f \cdot \zeta \nabla_x h\abramo{\,d\mathbf{x}dcdt}}_{=:I_2}\\
 & \qquad +\underbrace{\int_0^T\abramo{\int_0^1\int_{\Omega}}\left(\Sigma_{\theta}^{\theta}(f)-\Sigma(f)\right)(f_{\theta}-f)\cdot \zeta \nabla_x h\abramo{\,d\mathbf{x}dcdt}}_{=:I_3}+\underbrace{\int_0^T\abramo{\int_0^1\int_{\Omega}}\Sigma(f)\left(f_{\theta}-f\right)\cdot \zeta \nabla_x h\abramo{\,d\mathbf{x}dcdt}}_{=:I_4}\\
 & \qquad +\int_0^T\abramo{\int_0^1\int_{\Omega}}\Sigma(f)f\cdot \zeta \nabla_x h\abramo{\,d\mathbf{x}dcdt}.
\end{align*}
We observe that
\begin{align*}
    & |I_1|\leq \lVert \Sigma_{\theta}^{\theta}(f_{\theta})-\Sigma_{\theta}^{\theta}(f)\rVert_{\abramo{L^{2}}(0,T;L^{\infty}(\Omega \times (0,1);\mathbb{R}^d))}\lVert f_{\theta}\rVert_{L^{\infty}(0,T;L^1(\Omega \times (0,1)))}\lVert \zeta \rVert_{\abramo{L^2}(0,T)}\lVert \nabla_x h\rVert_{L^{\infty}(\Omega \times (0,1))}\\
    & \quad \leq C \lVert f_{\theta}-f\rVert_{\abramo{L^2}(0,T;L^{1}(\Omega \times (0,1)))}\lVert f_{\theta}\rVert_{L^{\infty}(0,T;L^1(\Omega \times (0,1)))}\lVert \zeta \rVert_{\abramo{L^2}(0,T)}\lVert \nabla_x h\rVert_{L^{\infty}(\Omega \times (0,1))}\to 0
\end{align*}
as $\theta \to 0$, thanks to property $(ii)$ of Lemma~\ref{lem1d} and the linearity of $\Sigma_{\theta}^{\theta}(\cdot)$ and to \eqref{conv15}.
Moreover,  recalling from \eqref{wp16} that
\[
\Sigma_{\theta}^{\theta}(\cdot)=G^{\theta}\Sigma_{\theta}(\cdot),
\]
we rewrite $I_2$ as
\[
I_2=\int_0^T\abramo{\int_0^1\int_{\Omega}}G^{\theta}\left(\Sigma_{\theta}(f)-\Sigma(f)\right)f \cdot \zeta \nabla_x h\abramo{\,d\mathbf{x}dcdt}+\int_0^T\abramo{\int_0^1\int_{\Omega}}\left(G^{\theta}-1\right)\Sigma(f)f\cdot \zeta \nabla_x h\abramo{\,d\mathbf{x}dcdt}.
\]
\abramo{Thanks to the property \eqref{gdws} of $G^{\theta}(\cdot)$ and the regularity of $f$ and of the test functions, we have that $\int_0^T\abramo{\int_0^1\int_{\Omega}}\left(G^{\theta}-1\right)\Sigma(f)f\cdot \zeta \nabla_x h\abramo{\,d\mathbf{x}dcdt}\to 0$ as $\theta \to 0$. For what concerns the first term, we start by observing that, thanks to \eqref{wp13}, we have that  $\left( P_{\theta,\Delta_1,\Delta_2} - P_{\Delta_1,\Delta_2} \right) (\vec{x} - \vec{x}^*) f(\vec{x}^*, c^*, t) \to 0$ for a.e. $(\vec{x}^*, c^*)\in \Omega\times (0,1)$. Moreover, thanks to \eqref{wp14}, we have that $\left| \left( P_{\theta,\Delta_1,\Delta_2} - P_{\Delta_1,\Delta_2} \right) (\vec{x} - \vec{x}^*) f(\vec{x}^*, c^*, t)\right|\leq 2 \, \Delta_1 |f(\vec{x}^*, c^*, t)|$. Then, \eqref{wp49} and the Dominated Convergence Theorem imply that $\Sigma_{\theta}(f)-\Sigma(f)\to 0$ a.e. in $\Omega \times (0,1) \times (0,T)$. This, together with property $(jv)$ of $G^{\theta}$, gives that the integrand $G^{\theta}\left(\Sigma_{\theta}(f)-\Sigma(f)\right)f\cdot \zeta \nabla_x h\to 0$ a.e. in $\Omega\times (0,1)\times (0,T)$. Also, from property $(i)$ 
of Lemma \ref{lem1} we have that 
\[
\left|G^{\theta}\left(\Sigma_{\theta}(f)-\Sigma(f)\right)f\cdot \zeta \nabla_x h\right|\leq 2\Delta_1 \lVert f\rVert_{L^\infty(0,T;L^1(\Omega\times (0,1)))} \lVert\nabla_x h\rVert_{L^\infty(\Omega \times (0,1))} \lVert\zeta\rVert_{L^\infty(0,T)}|f(\vec{x},c,t)|.
\]
Again, \eqref{wp49}, the regularity of the test functions and the Dominated Convergence Theorem imply that $\int_0^T\abramo{\int_0^1\int_{\Omega}}G^{\theta}\left(\Sigma_{\theta}(f)-\Sigma(f)\right)f\cdot \zeta \nabla_x h\abramo{\,d\mathbf{x}dcdt}\to 0$ as $\theta \to 0$. Hence,} \pier{it turns out that}
\[
I_2\to 0
\]
as $\theta \to 0$. For what concerns $I_3$, we have that
\begin{align*}
    & |I_3|\leq \lVert \Sigma_{\theta}^{\theta}(f)-\Sigma(f)\rVert_{L^{\infty}(0,T;L^{\infty}(\Omega \times (0,1);\mathbb{R}^d))}\lVert f_{\theta}-f\rVert_{\abramo{L^2}(0,T;L^1(\Omega \times (0,1)))}\lVert \zeta \rVert_{\abramo{L^2}(0,T)}\lVert \nabla_x h\rVert_{L^{\infty}(\Omega \times (0,1))}\to 0
\end{align*}
as $\theta \to 0$, thanks to properties $(i)$ of Lemma~\ref{lem1} and $(ii)$ of Lemma~\ref{lem1d} and thanks to \eqref{conv15}.
Finally, we have that
\begin{align*}
    & |I_4|\leq \lVert \Sigma(f)\rVert_{L^{\infty}(0,T;L^{\infty}(\Omega \times (0,1);\mathbb{R}^d))}\lVert f_{\theta}-f\rVert_{\abramo{L^2}(0,T;L^1(\Omega \times (0,1)))}\lVert \zeta \rVert_{\abramo{L^2}(0,T)}\lVert \nabla_x h\rVert_{L^{\infty}(\Omega \times (0,1))}\to 0
\end{align*}
as $\theta \to 0$, thanks to property $(i)$ of Lemma~\ref{lem1} and to \eqref{conv15}.

We thus conclude that
\[
\int_0^T(\Sigma_{\theta}^{\theta}(f_{\theta})f_{\theta}, \zeta \nabla_x h)\abramo{\,dt}\to \int_0^T\abramo{\int_0^1\int_{\Omega}}\Sigma(f)f\cdot \zeta \nabla_x h\abramo{\,d\mathbf{x}dcdt}
\]
as $\theta \to 0$.

\abramo{All the other terms in \eqref{wp46lim} pass to the limit in a straightforward way thanks to \eqref{conv15}, and thus we recover \pier{\eqref{wp50} in the limit}. This concludes the proof.}
\end{proof}
\abramo{
\begin{remark}
    Due to the low regularity of the solutions of \eqref{wp50}, their uniqueness remains an open problem. We remark that at the level of the regularized systems we could establish both uniqueness and continuous dependence \pier{of the solution on the data}, which is important in view of numerical approximations.
\end{remark}
}

\section*{Conclusions}
\mz{
We have established existence of global distributional solutions for a kinetic model of consensus-based image segmentation on a bounded spatial domain. The model combines a nonlocal Hegselmann–Krause interaction with a bounded-confidence kernel and a diffusion coefficient that degenerates at the endpoints of the gray-level interval. These \er{aspects}, together with the absence of diffusion in the feature variable, prevent a direct application of standard compactness arguments.
Our analysis relies on a three-level regularization of the interaction kernel, the nonlocal drift near the spatial boundary, and the degenerate diffusion. We first establish well-posedness and regularity for \pc{the solution to} the regularized problems, then pass to the limit as the regularization parameters vanish. The resulting distributional solution is nonnegative and preserves total mass. This extends the analytical framework for consensus-based segmentation models to bounded domains and bounded-confidence interactions \aa{characterized by jump discontinuities}.
Future work will focus on the large-time behaviour of solutions, in particular whether and under which conditions the kinetic dynamics lead to the formation of clusters that are compatible with available segmentation masks. A further step will be to compare the segmentation masks obtained from the model with those derived from real images, assessing how the interaction thresholds and diffusion affect the resulting segmentation.}

\section*{Acknowledgments}
\mz{The research underlying this paper has been undertaken within the activities of the GNAMPA and GNFM groups
of INdAM (National Institute of High Mathematics). All the authors acknowledge partial support from the PRIN2022PNRR project No.P2022Z7ZAJ, European Union - NextGenerationEU. M.Z. acknowledges partial support by ICSC - Centro Nazionale di Ricerca in High Performance Computing, Big Data and Quantum Computing, funded by European Union - NextGenerationEU.}

\end{document}